\documentclass[11pt,twoside]{amsart}
\usepackage{latexsym,amssymb,amsmath}

\numberwithin{equation}{section}
\newtheorem{theorem}{Theorem}[section]
\newtheorem{lemma}[theorem]{Lemma}
\newtheorem{proposition}[theorem]{Proposition}
\newtheorem{corollary}[theorem]{Corollary}

\theoremstyle{definition}
\newtheorem{definition}[theorem]{Definition} 
 
\newtheorem{remark}[theorem]{Remark}
\newtheorem{example}[theorem]{Example}

\begin{document}


\title[Algebraic methods
for parameterized codes]{Algebraic methods for parameterized codes
and invariants of vanishing
ideals over finite fields} 

\author{Carlos Renter\'\i a-M\'arquez}
\address{
Departamento de Matem\'aticas\\
Escuela Superior de F\'\i sica y
Matem\'aticas\\
Instituto Polit\'ecnico Nacional\\
07300 Mexico City, D.F.}
\email{renteri@esfm.ipn.mx}
\thanks{The first author was partially supported by 
COFAA-IPN and SNI. The second author was partially supported by a grant
of CNPq. The third author was partially supported by CONACyT 
grant 49251-F and SNI}

\author{Aron Simis}
\address{
Departamento de Matem\'atica\\
Universidade Federal
de Pernambuco\\
50740-540 Recife\\ Pe\\ Brazil
}\email{aron@dmat.ufpe.br}

\author{Rafael H. Villarreal}
\address{
Departamento de
Matem\'aticas\\
Centro de Investigaci\'on y de Estudios
Avanzados del
IPN\\
Apartado Postal
14--740 \\
07000 Mexico City, D.F.
}
\email{vila@math.cinvestav.mx}

\subjclass[2000]{Primary 13P25; Secondary 14G50, 14G15, 11T71, 94B27, 94B05.} 

\begin{abstract} Let $K=\mathbb{F}_q$ be a finite field with $q$
elements and let $X$ be a subset of a projective space ${\mathbb
P}^{s-1}$, over the field $K$, parameterized by Laurent monomials. 
Let $I(X)$ be the vanishing ideal of $X$. Some
of the main contributions of this paper are in determining 
the structure of $I(X)$ to compute some of its invariants. It is
shown that $I(X)$ is a lattice ideal. We introduce the notion of a 
parameterized code arising from
$X$ and present algebraic methods to compute and study its dimension,
length and minimum distance. 
For a parameterized code, arising from a
connected graph, we are able to compute its length and to make our
results more precise.  
If the graph
is non-bipartite, we show an upper bound for the minimum
distance. 
\end{abstract}

\maketitle

\section{Introduction}

Let $K=\mathbb{F}_q$  be a finite field with $q$ elements and 
let $y^{v_1},\ldots,y^{v_s}$ be a finite set of Laurent
monomials.  Given $v_i=(v_{i1},\ldots,v_{in})\in\mathbb{Z}^n$, we set 
$$
y^{v_i}=y_1^{v_{i1}}\cdots y_n^{v_{in}},\ \ \ \ i=1,\ldots,s,
$$
where $y_1,\ldots,y_n$ are the indeterminates of a ring of Laurent
polynomials with coefficients in $K$. 

An object of study here is the following set
parameterized  by these monomials 
$$
X:=\{[(x_1^{v_{11}}\cdots x_n^{v_{1n}},\ldots,x_1^{v_{s1}}\cdots
x_n^{v_{sn}})]\, \vert\, x_i\in K^*\mbox{ for all
}i\}\subset\mathbb{P}^{s-1},
$$
where $K^*=K\setminus\{0\}$ and $\mathbb{P}^{s-1}$ is a projective
space over the field $K$. Following \cite{afinetv} we call $X$ an 
{\it algebraic toric set\/} parameterized  by
$y^{v_1},\ldots,y^{v_s}$. We are especially interested in measuring 
the size of $X$, in terms of $q$, $n$ and $s$, because $|X|$ is the
length of the linear codes that we will introduce and examine here.
    
Let $S=K[t_1,\ldots,t_s]=\oplus_{d=0}^\infty S_d$ be a polynomial ring 
over the field $K$ with the standard grading. Another object of study
is the graded ideal $I(X)\subset S$ generated by the  
homogeneous polynomials of $S$ that vanish on $X$. The ideal $I(X)$ is
called the {\it vanishing ideal\/} of $X$.    

Some of the main contributions of this paper are in determining 
the structure of $I(X)$ to compute some of its invariants. The other main
contributions are estimates (in certain cases 
formulas) of the basic parameters of certain linear codes. 

The main application we foresee is to algebraic coding theory because
our results can be used to study the performance of a new class of
evaluation codes that 
we now introduce. Let $[P_1],\ldots,[P_m]$ be the points of $X$ and
let $f_0(t_1,\ldots,t_s)=t_1^d$. The {\it 
evaluation map\/} 
$$
{\rm ev}_d\colon S_d=K[t_1,\ldots,t_s]_d\rightarrow K^{|X|},\ \ \ \ \ 
f\mapsto \left(\frac{f(P_1)}{f_0(P_1)},\ldots,\frac{f(P_m)}{f_0(P_m)}\right)
$$
defines a linear map of
$K$-vector spaces. The image of ${\rm ev}_d$, denoted by $C_X(d)$,
defines a {\it linear code} that we call a {\it parameterized code\/} of
order $d$. By a {\it linear code\/} we mean a linear subspace of
$K^{|X|}$. The kernel of ${\rm ev}_d$ is the homogeneous part
$I(X)_d$ of degree $d$ of $I(X)$. Therefore there is an isomorphism of $K$-vector spaces
$$S_d/I(X)_d\simeq C_X(d).$$

The {\it dimension\/} and the {\it length\/} of $C_X(d)$ are given by
$\dim_K C_X(d)$ and $|X|$ 
respectively. We will provide algebraic
methods to compute and study the dimension and the 
length of $C_X(d)$, which are two of the basic parameters of a linear
code. A third basic parameter is the {\it minimum
distance\/} of $C_X(d)$, which is given by 
$\delta_d=\min\{\|v\|
\colon 0\neq v\in C_X(d)\}$, where $\|v\|$ is the number of non-zero
entries of $v$. The basic parameters of $C_X(d)$ are related by the 
Singleton bound for the minimum distance:
$$
\delta_d\leq |X|-\dim_KC_X(d)+1.
$$

A good parameterized code should have large $|X|$ and  
with $\dim_KC_X(d)/|X|$ and $\delta_d/|X|$ as large as possible. 
Evaluation codes associated to a projective torus are called {\it 
generalized Reed-Solomon\/} codes \cite{GRH}. Parameterized
codes are a natural 
extension of this sort of codes. Some special families of
evaluation codes 
have been extensively studied, including several 
variations of Reed-Muller codes 
\cite{duursma-renteria-tapia,gold-little-schenck,GR,GRT,lachaud,renteria-tapia-ca}.

Two of the basic parameters of $C_X(d)$ can be expressed 
using Hilbert functions of standard graded algebras \cite{Sta1} as we
now  explain. Recall that the
{\it Hilbert function\/} of
$S/I(X)$ is given by 
$$H_X(d):=\dim_K\, 
(S/I(X))_d=\dim_K\, 
S_d/I(X)_d=\dim_KC_X(d).$$

The unique polynomial $h_X(t)\in \mathbb{Z}[t]$ such that $h_X(d)=H_X(d)$ for
$d\gg 0$ is called the {\it Hilbert polynomial\/} of  
$S/I(X)$. In our situation $h_X(t)$ is a non-zero constant.
Furthermore $h_X(d)=|X|$ for $d\geq |X|-1$, see \cite[Lecture
13]{harris}. This means that $|X|$ equals the {\it degree\/} 
of $S/I(X)$. Thus $H_X(d)$ and ${\rm deg}\, S/I(X)$ equal the
dimension and the length of $C_X(d)$ respectively.

The results of this paper will allow to compute the dimension
and the length of $C_X(d)$ using Hilbert functions. In certain
interesting cases we show a nice 
formula for the length. For algebraic toric sets 
arising from combinatorial structures, we are able to estimate the
length in terms of $n$, $q$, and the rank of a certain subgroup of
$\mathbb{Z}^{n+1}$. When $C_X(d)$ arises from a connected non-bipartite graph, we
will show an upper bound for the minimum distance and
compare this bound with the Singleton bound (see
Section~\ref{minimum-distance-section}).  

The contents of this paper are as follows. The main theorems in 
Section \ref{section-parameterized} are algebraic expressions for
$I(X)$, which can be used to extract information about the basic
parameters of $C_X(d)$ using
Gr\"obner bases. Before introducing the 
theorems, recall that an additive subgroup of $\mathbb{Z}^s$ 
is called a {\it lattice\/}. A {\it lattice
ideal\/} of $S$ 
is an 
ideal of the form
$$
I(\mathcal{L}):=(\{t^{a}-t^{b}\vert\, 
a,b\in\mathbb{N}^s\mbox{ with }a-b\in\mathcal{L}\})\subset S
$$
for some lattice $\mathcal{L}\subset \mathbb{Z}^s$. 
A polynomial of the form $t^a-t^b$,
with $a,b\in\mathbb{N}^s$, is called a {\it binomial}  
of $S$. An ideal generated 
by binomials is called a {\it binomial ideal\/}. The concept of
a lattice ideal is a natural generalization of a toric ideal
\cite[Corollary~7.1.4]{monalg}.  

In Theorem~\ref{ipn-ufpe-cinvestav} we show that $I(X)$ is a 
radical Cohen-Macaulay lattice ideal of dimension $1$. Moreover, if 
$v_i\in\mathbb{N}^n$ for all $i$, we prove the equality 
$$
I(X)=(t_1-y^{v_1}z,\ldots,t_s-y^{v_s}z,
y_1^{q-1}-1,\ldots,y_n^{q-1}-1)\cap S,
$$
where $z$ is a new indeterminate. A similar statement holds for
arbitrary $v_i$'s (see Theorem~\ref{ipn-ufpe-cinvestav-1}). 
In light of this result, we can compute the reduced Gr\"obner 
basis of $I(X)$, with respect to any term order of the monomials of
$S$, using the  
computer algebra system {\em Macaulay\/}$2$
\cite{computations-macaula2,mac2}. Thus, we can compute the Hilbert
function and the degree of $S/I(X)$, i.e., we can compute the  
dimension and the length of $C_X(d)$.

We present a different expression for $I(X)$---via a
saturation process---valid for a wide class
of algebraic toric sets (see Theorem~\ref{vila-dictaminadora} and
Corollary~\ref{sunday-morning-06-09-09}).   
As a consequence, if 
$$
\mathbb{T}=\{[(x_1,\ldots, x_s)]\vert\, x_i\in
K^*\mbox{ for all }i\}
\subset\mathbb{P}^{s-1}
$$ 
is a {\it projective torus\/}, then
$I(\mathbb{T})=(\{t_i^{q-1}-t_1^{q-1}\}_{i=2}^s)$ (see
Corollary~\ref{sept15-09}). This equality was first 
shown in \cite{GRH}. Then we obtain a family of algebraic
toric sets---arising from connected graphs---where $I(X)$ can be
computed using a saturation process (see 
Corollary~\ref{sept29-09-2}).  

In Section~\ref{degree-of-coordinate-ring} we focus on the computation
of $|X|$, the length of $C_X(d)$. We uncover a direct method, 
based on integer programming techniques, to compute $|X|$ 
(see Proposition~\ref{sept29-09}). Under certain conditions
we prove that $(q-1)^{r-1}$ divides the length of $C_X(d)$, where $r$
is the rank of the subgroup generated by $\{(v_i,1)\}_{i=1}^s$ (see
Theorem~\ref{sept23-09}).  In some cases---when $X$ comes from a
connected graph---we give a formula for the 
length of $C_X(d)$ (see Corollary~\ref{sept29-09-1}). 

The elements of $C_X(d)$ can be interpreted as rational functions 
on $X$. For this reason, in Section~\ref{vanishing-ideal-projective},
we study  
the geometric
structure of $X$. Let $I_\mathcal{A}$ be the {\it toric ideal\/} of
$\mathcal{A}=\{v_1,\ldots,v_s\}$,  
i.e., $I_\mathcal{A}$ is the prime ideal
of $S$ of polynomial relations of $y^{v_1},\ldots,y^{v_s}$. We call
$\mathcal{A}$ {\it homogeneous\/} if $\mathcal{A}$ lies on an affine 
hyperplane not containing the origin. We prove that if
$\mathcal{A}$ is homogeneous, then the projective toric variety
$V(I_\mathcal{A})$ intersected with a projective torus
$\mathbb{T}\subset\mathbb{P}^{s-1}$ is always parameterized  by
Laurent monomials (see Theorem~\ref{sept29-09-3}(i)). This gives a method to produce projective
varieties parameterized  by Laurent monomials. As a byproduct, letting
$V_\mathcal{A}$ denote $V(I_\mathcal{A})\cap\mathbb{T}$, our results
allow to 
compute $I({V_\mathcal{A}})$ using Gr\"obner bases (see
Theorem~\ref{sept29-09-3}(ii)). As we will see, often an algebraic toric set $X$ 
is in fact a projective variety defined by binomials (see
Proposition~\ref{sept29-09-5}). In particular, we obtain the equality $X=V_\mathcal{A}$ for any
$\mathcal{A}$ arising from the edges of a connected graph. As a consequence, 
we show a finite Nullstellensatz (see Corollary~\ref{finite-nullstellensatz}). 

The dimension of $C_X(d)$ is increasing, as a function of 
$d$, until it reaches a constant value
\cite{duursma-renteria-tapia,geramita-cayley-bacharach}. We observe
that the minimum distance of $C_X(d)$ has the opposite behaviour: it is decreasing, as a
function of $d$, until it reaches a constant value (see
Proposition~\ref{minimum-distance-behaviour}).

Finally, in Section~\ref{minimum-distance-section}, we present an
application of our results and techniques to algebraic coding theory.
We show upper bounds for the minimum 
distance of parameterized codes arising from  
a connected non-bipartite graph (see Theorem~\ref{carlos-aron-vila}).
The geometric 
perspective of 
Section~\ref{vanishing-ideal-projective} plays a role
here. A comparison between our bound
and the Singleton bound is given (see
Remark~\ref{comparison-remark} and Example~\ref{k5}). We give an
explicit formula for the minimum distance of  
$C_X(d)$ when $X$ is a projective torus in $\mathbb{P}^2$ (see
Proposition~\ref{minimum-distance-p2}). Part of this formula was
already known \cite{GRH}; our contribution here is to use a result of
\cite{hansen} together with the proof of
Theorem~\ref{carlos-aron-vila} 
to treat the cases not covered in \cite{GRH}.

For all unexplained
terminology and additional information,  we refer to
\cite{cca,Stur1} (for the theory of binomial and toric ideals),
\cite{Eisen,Vas1} (for computational commutative algebra), 
\cite{Boll} (for graph theory), and
\cite{MacWilliams-Sloane,stichtenoth,tsfasman} (for the theory of
error-correcting codes and linear codes).

\section{The ideal of an algebraic toric set parameterized  by
monomials}\label{section-parameterized}

We continue to use the notation and definitions used in the
introduction. Here we study the structure of the graded
ideal $I(X)$ and show algebraic methods to compute a finite set of
binomials generating $I(X)$. We begin this section by introducing $X^*$,
the affine companion of $X$, that shares some of the properties of
$X$, such as being a multiplicative group. Some of our results
will admit  affine versions for $X^*$ as well. However, as a matter 
of staying focused, we will deal mostly with $X$ while $X^*$ will play
by and large an auxiliary role.

Let $K=\mathbb{F}_q$  be a finite field with $q$ elements and let 
$K[y_1^{\pm 1},\ldots, y_n^{\pm 1}]$ be a ring of Laurent polynomials
with coefficients in $K$. Consider a set $y^{v_1},\ldots,y^{v_s}$ of
Laurent monomials with $v_i\in\mathbb{Z}^n$ and
$v_i=(v_{i1},\ldots,v_{in})\in\mathbb{Z}^n$. The following set is called the 
{\it affine algebraic toric set\/} parameterized  by these monomials:   
$$
X^*:=\{(x_1^{v_{11}}\cdots x_n^{v_{1n}},\ldots,x_1^{v_{s1}}\cdots
x_n^{v_{sn}})\vert\, x_i\in K^*
\mbox{ for all }i\}.
$$
This model of parametrization was introduced in
\cite{afinetv}. In  \cite{katsabekis-thoma,afinetv} a classification
of the affine toric 
varieties that are parameterized by monomials is given. The set
$(K^*)^s$ is called an {\it affine algebraic torus} of
dimension $s$ and is denoted  by
$\mathbb{T}^*$. The affine torus $\mathbb{T}^*$ is a multiplicative group under
the product operation 
$$\alpha\cdot\alpha'=(\alpha_1,\ldots,\alpha_s)\cdot(\alpha'_1,\ldots,\alpha'_s)=
(\alpha_1\alpha'_1,\ldots,\alpha_s\alpha'_s).$$ 
Clearly, the set $X^*$ is also a group under componentwise multiplication. We have the
inclusions $X^*\subset 
\mathbb{T}^*\subset\mathbb{A}^s\setminus\{0\}$, where $\mathbb{A}^s$
denotes the 
affine space $K^s$.

The {\it projective space\/} of 
dimension $s-1$ over $K$, denoted by 
$\mathbb{P}^{s-1}$, is the quotient space 
$$(K^{s}\setminus\{0\})/\sim $$
where two points $\alpha$, $\beta$ in $K^{s}\setminus\{0\}$ 
are equivalent if $\alpha=\lambda{\beta}$ for some $\lambda\in K$. We
denote the  
equivalence class of $\alpha$ by $[\alpha]$. By
definition, there is a structure map
$$
\varphi_s\colon\mathbb{A}^{s}\setminus\{0\}\longrightarrow \mathbb{P}^{s-1},\, 
\ \ \ \ \ \ \alpha\longmapsto [\alpha].
$$

The image of $X^*$ under $\varphi_s$ will be denoted by $X$. The set 
$X$ is the {\it algebraic toric set\/} parameterized  by
$y^{v_1},\ldots,y^{v_s}$ that was defined earlier in the introduction:
$$
X:=\{[(x_1^{v_{11}}\cdots x_n^{v_{1n}},\ldots,x_1^{v_{s1}}\cdots
x_n^{v_{sn}})]\, \vert\, x_i\in K^*\mbox{ for all
}i\}\subset\mathbb{P}^{s-1}.
$$
The set $X$ is a multiplicative group with the product 
operation induced by that of $X^*$. The group structure of $X$ and $X^*$ 
will come into play in Section~\ref{degree-of-coordinate-ring}.

Let $S=K[t_1,\ldots,t_s]$ be a polynomial ring 
with coefficients in the field $K$ with the standard grading
$S=\oplus_{d=0}^\infty S_d$ induced by 
setting $\deg(t_i)=1$ for all
$i$. We are interested in the radical ideal $I(X)$ generated by the 
homogeneous polynomials of $S$ that vanish on $X$. 

Recall the following notion from commutative ring theory, which will
be used a few times in the exposition. Let $D$ be a commutative ring with unit and let
$M$ be a $D$-module. The set 
$$
\mathcal{Z}_D(M):=\{r\in D\, \vert\, rm=0\mbox{ for some }0\neq m\in
M\}
$$ 
is called the set of {\it zero divisors} of $M$. If $D$ is the ring of 
integers, we denote the set of zero divisors of $M$ simply by
$\mathcal{Z}(M)$.

We come to one of the main results of this section, an structure
theorem allowing---with the help of {\em
Macaulay\/}$2$ \cite{computations-macaula2,mac2}---the computation of the
Hilbert function and the degree of $S/I(X)$. 

\begin{theorem}\label{ipn-ufpe-cinvestav} 
Let $B=K[t_1,\ldots,t_s,y_1,\ldots,y_n,z]$ be 
a polynomial ring over the finite field $K=\mathbb{F}_q$. If
$v_i\in\mathbb{N}^n$ for all 
$i$, then the following holds{\rm :}
\begin{enumerate}
\item[\rm (a)]
$I(X)=(\{t_i-y^{v_i}z\}_{i=1}^s\cup\{y_i^{q-1}-1\}_{i=1}^n)\cap S$ and
$I(X)$ is a binomial ideal.
\item[\rm (b)] $t_i\notin\mathcal{Z}_S(S/I(X))$ for all $i$ and 
$I(X)$ is a radical lattice ideal.
\item[\rm (c)] $S/I(X)$ is a Cohen-Macaulay ring of dimension $1$.
\end{enumerate}
\end{theorem}

\begin{proof} (a) We set $I'=(t_1-y^{v_1}z,\ldots,t_s-y^{v_s}z,
y_1^{q-1}-1,\ldots,y_n^{q-1}-1)\subset B$. First we show the inclusion
$I(X)\subset I'\cap S$. Take a homogeneous polynomial
$F=F(t_1,\ldots,t_s)$ of degree $d$ that vanishes on $X$. We can write
\begin{equation}\label{aug27-09}
F=\lambda_1 t^{m_1}+\cdots+\lambda_r t^{m_r}\ \ \ \
(\lambda_i\in K^*;\, m_i\in \mathbb{N}^s), 
\end{equation}
where $\deg(t^{m_i})=d$ for all $i$. Write
$m_i=(m_{i1},\ldots,m_{is})$ for $1\leq i\leq r$. Applying the binomial
theorem to expand the right hand side of the equality
$$
t_j^{m_{ij}}=\left[(t_j-y^{v_j}z)+y^{v_j}z\right]^{m_{ij}},\ \ \ 
1\leq i\leq r,\ 1\leq j\leq s,
$$
and then substituting all the $t_j^{m_{i j}}$ in Eq.~(\ref{aug27-09}), we obtain
that $F$ can 
be written as:
\begin{equation}\label{23-jul-10}
F=\sum_{i=1}^sg_i(t_i-y^{v_i}z)+z^dF(y^{v_1},\ldots,y^{v_s})
\end{equation}
for some $g_1,\ldots,g_s$ in $B$. By the division algorithm
in $K[y_1,\ldots,y_n]$ (see \cite[Theorem~3, p.~63]{CLO}) we can
write 
\begin{equation}\label{23-jul-10-1}
F(y^{v_1},\ldots,y^{v_s})=\sum_{i=1}^nh_i(y_i^{q-1}-1)+G(y_1,\ldots,y_n)
\end{equation}
for some $h_1,\ldots,h_n$ in $K[y_1,\ldots,y_n]$, 
where the monomials that occur in $G=G(y_1,\ldots,y_n)$ are not divisible by 
any of the monomials $y_1^{q-1},\ldots,y_n^{q-1}$, i.e.,
$\deg_{y_i}(G)<q-1$ for $i=1,\ldots,n$.
Therefore, using Eqs.~(\ref{23-jul-10}) and (\ref{23-jul-10-1}), we
obtain the equality
\begin{equation}\label{23-jul-10-2}
F=\sum_{i=1}^sg_i(t_i-y^{v_i}z)+\left(\sum_{i=1}^nh_i(y_i^{q-1}-1)\right)z^d+
G(y_1,\ldots,y_n)z^d.
\end{equation}
Thus to show that $F\in I'\cap S$ we need only show that $G=0$. We
claim that  
$G$ vanishes on $(K^*)^n$. Take an arbitrary sequence $x_1,\ldots,x_n$
of elements of $K^*$. Making $t_i=x^{v_i}$ for all $i$ in
Eq.~(\ref{23-jul-10-2}) and using that
$F$ vanishes on $X$, we obtain
\begin{equation}\label{23-jul-10-3}
0=F(x^{v_1},\ldots,x^{v_s})=\sum_{i=1}^sg_i'(x^{v_i}-y^{v_i}z)+
\left(\sum_{i=1}^nh_i(y_i^{q-1}-1)\right)z^d+
G(y_1,\ldots,y_n)z^d.
\end{equation}
Since $(K^*,\,\cdot\, )$ is a group of order $q-1$, 
we can then make $y_i=x_i$ for all $i$ and $z=1$ in
Eq.~(\ref{23-jul-10-3}) to get that $G$
vanishes on $(x_1,\ldots,x_n)$. This completes the proof of the 
claim. Therefore $G$ vanishes on $(K^*)^n$ and $\deg_{y_i}(G)<q-1$ 
for all $i$. By induction on $n$ it follows that $G=0$. We can also
show that $G=0$ by a direct application of the combinatorial Nullstellensatz
\cite{alon-cn}.  

Next we show the inclusion $I(X)\supset I'\cap S$. Let $\mathcal{G}$
be a Gr\"obner basis of  
$I'$ with respect to the lexicographic order $y_1\succ\cdots\succ y_n\succ
z\succ t_1\succ\cdots\succ t_s$. By Buchberger algorithm
\cite[Theorem~2, p.~89]{CLO} the set $\mathcal{G}$ consists of binomials and 
by elimination theory \cite[Theorem~2, p.~114]{CLO} the set
$\mathcal{G}\cap S$ is a Gr\"obner 
basis of $I'\cap S$. 
Hence $I'\cap S$ is a binomial ideal. Thus to show the inclusion
$I(X)\supset I'\cap S$ it suffices to show that any binomial in
$I'\cap S$ is homogeneous and 
vanishes on $X$. Take a binomial $f=t^a-t^b$ in $I'\cap S$, where
$a=(a_i)$ and $b=(b_i)$ 
are in $\mathbb{N}^s$. Then we can
write
\begin{equation}\label{sept1-09}
f=\sum_{i=1}^sg_i(t_i-y^{v_i}z)+\sum_{i=1}^nh_i(y_i^{q-1}-1)
\end{equation}
for some polynomials $g_1,\ldots,g_s, h_1,\ldots,h_n$ in $B$. Making $y_i=1$ for
$i=1,\ldots,n$ and $t_i=y^{v_i}z$ for 
$i=1,\ldots,s$, we get
$$
z^{a_1}\cdots z^{a_s}-z^{b_1}\cdots z^{b_s}=0\ \Longrightarrow\ 
a_1+\cdots+a_s=b_1+\cdots+b_s.
$$
Hence $f$ is homogeneous. Take a point 
$[P]$ in $X$ with $P=(x^{v_1},\ldots,x^{v_s})$. Making $t_i=x^{v_i}$ in
Eq.~(\ref{sept1-09}), we get 
$$
f(x^{v_1},\ldots,x^{v_s})=\sum_{i=1}^sg_i'(x^{v_i}-y^{v_i}z)+
\sum_{i=1}^nh_i'(y_i^{q-1}-1).
$$
Hence making $y_i=x_i$ for all $i$ and $z=1$, we get that 
$f(P)=0$. Thus $f$ vanishes on $X$.

Thus, we have shown the equality $I(X)=I'\cap S$. The
proof of the inclusion $I(X)\supset I'\cap S$ shows that $I'\cap S$
is a binomial ideal. Hence $I(X)$ is a binomial ideal.

(b) Observe that a binomial ideal $J\subset S$ is a lattice ideal if and only
if  $t_i\notin\mathcal{Z}_S(S/J)$ for all $i$. This is a consequence of
\cite[Corollary~2.5]{EisStu}. Thus by part (a) we need only show 
that $t_i$ is not a zero divisor of $S/I(X)$ for all $i$. Let
$[P]$ be a point in $X$, with $P=(\alpha_1,\ldots,\alpha_s)$, and let
$I_{[P]}$ be the 
ideal generated by the homogeneous polynomials of $S$ that vanish 
at $[P]$. Then 
\begin{equation}\label{primdec-ix}
I_{[P]}=(\alpha_1t_2-\alpha_2t_1,\alpha_1t_3-\alpha_3t_1,\ldots,
\alpha_1t_s-\alpha_st_1)\ \mbox{ and }\ I(X)=\bigcap_{[P]\in X}I_{[P]}
\end{equation}
and the later is the primary decomposition of $I(X)$, because $I_{[P]}$ is
a prime ideal of $S$ for any $[P]\in X$. Thus ${\rm rad}\, I(X)=I(X)$, i.e., 
$I(X)$ is a radical ideal. Since
$$\mathcal{Z}_S(S/I(X))=\bigcup_{[P]\in X} I_{[P]}$$ 
it is seen that $t_i$ is not a
zero divisor for any $i$. 

(c) As $I_{[P]}$ has height $s-1$ for any $[P]\in X$, we get that $\dim S/I(X)=1$. By
(b) any variable $t_i$ is a $S$-regular element of $S/I(X)$. Thus any
variable $t_i$ form a homogeneous regular system of parameters of $S/I(X)$,
i.e., $S/I(X)$ is a Cohen-Macaulay ring by \cite[Proposition~2.2.7]{monalg}.
 \end{proof}

By Theorem~\ref{ipn-ufpe-cinvestav}(a), the ideal $I(X)$ is generated by binomials. This fact is
surprising, because according to Eq.~(\ref{primdec-ix}) $I(X)$ is a radical ideal
and all its minimal primes, except 
$\mathfrak{p}=(\{t_i-t_1\}_{i=2}^s)$, are non-binomial. 

The next notion that we need is that of the saturation of an ideal with respect to
a polynomial. We will determine when $I(X)$ can be obtained by a
saturation process (see Corollary~\ref{sunday-morning-06-09-09}). 

\begin{definition}
For an ideal $Q\subset S$ and  a polynomial $h\in S$, the {\it
saturation\/} of $Q$ with respect to $h$
is the ideal
$$(Q\colon
h^{\infty}):=\{f\in S\vert\, fh^m\in Q\mbox{ for some }m\geq 1\}.
$$
We will only deal with the case where $h=t_1\cdots t_s$. 
\end{definition}

Let $\mathcal{A}=\{v_1,\ldots,v_s\}\subset\mathbb{Z}^n$ and let 
$I_\mathcal{A}$ be its associated  
{\it toric ideal\/}, i.e., $I_\mathcal{A}$ is the prime ideal of
$S$ given by (see \cite{Stur1}):
\begin{equation}\label{oct12-09}
I_\mathcal{A}=\left.\left(
t^a-t^b\right\vert\, a=(a_i),b=(b_i)\in\mathbb{N}^s,\textstyle\sum_ia_iv_i=\sum_i
b_iv_i\right)\subset S.
\end{equation}
The toric ideal $I_\mathcal{A}$ is the kernel of
the following epimorphism of $K$-algebras
$$
K[t_1,\ldots,t_s]\longrightarrow K[y^{v_1},\ldots,y^{v_s}]
$$
induced by $t_i\mapsto y^{v_i}$. We call
$\mathcal{A}$ {\it homogeneous\/} 
if there is a vector $x_0\in\mathbb{Q}^n$ such that $\langle
v_i,x_0\rangle=1$ for 
all $i$. From Eq.~(\ref{oct12-09}) it follows
that any binomial in $I_\mathcal{A}$ vanishes on $X$. If $\mathcal{A}$
is homogeneous, then any binomial in $I_\mathcal{A}$ is homogeneous, in
the standard grading of $S$, hence belongs to $I(X)$. The 
binomial $t_i^{q-1}-t_1^{q-1}$ vanishes on $(K^*)^s$ 
because $(K^*,\,\cdot\, )$ is a group of order $q-1$. Hence $t_i^{q-1}-t_1^{q-1}$ 
belongs to $I(X)$ for all $i$. Thus if $\mathcal{A}$ is homogeneous, 
then $I(X)$ contains the binomial ideal
$Q=I_\mathcal{A}+(\{t_i^{q-1}-t_1^{q-1}\}_{i=2}^s)$. For a large class of 
algebraic toric sets, we show that $I(X)$ is the saturation of $Q$ with
respect to $t_1\cdots t_s$. We also describe when $I(X)$ is the saturation of $Q$ with
respect to $t_1\cdots t_s$.

Let us introduce some more notation. Given $\Gamma\subset
\mathbb{Z}^n$, the subgroup of $\mathbb{Z}^n$ 
generated by $\Gamma$ is denoted by $\mathbb{Z}\Gamma$. 

\begin{lemma}\label{sept4-09-1}
If $c=(c_i)\in\mathbb{Z}^s$ and $\sum_ic_i=0$, then 
$c$ is in $\mathbb{Z}\{e_2-e_1,\ldots,e_s-e_1\}$, 
where $e_i$ is the $i${\it th} unit vector of $\mathbb{Z}^s$.
\end{lemma}

\begin{proof} Notice that
$\mathbb{Z}\{e_2-e_1,\ldots,e_s-e_1\}+\mathbb{Z}e_1=\mathbb{Z}^s$.
Then $c=\lambda_{1}e_1+\sum_{i=2}^s\lambda_i(e_i-e_1)$ for
some $\lambda_i\in\mathbb{Z}$. 
As $\sum_ic_i=0$, we get $\lambda_1=0$. 
\end{proof}

\begin{definition} Given $a\in {\mathbb R}^s$, its {\it
support\/} is defined as ${\rm supp}(a)=\{i\, |\, a_i\neq 0\}$. 
Note that $a$ can be uniquely written as $a=a^+-a^-$, 
where $a^+$ and $a^-$ are two non-negative vectors 
with disjoint support which are called the {\it positive\/} and 
{\it negative\/} part of $a$ respectively. 
\end{definition}

We come to another of the main results of this section. 

\begin{theorem}\label{vila-dictaminadora} Let $K=\mathbb{F}_q$ be a finite field, let
$\mathcal{A}=\{v_1,\ldots,v_s\}\subset\mathbb{Z}^n$, and let
$\phi\colon
\mathbb{Z}^n/L\rightarrow\mathbb{Z}^n/L$
be the multiplication map $\phi(\overline{a})=(q-1)\overline{a}$,
where $L=\mathbb{Z}\{v_i-v_1\}_{i=2}^s$. If
$\mathcal{A}$ is homogeneous, then
\begin{equation}\label{happy-sept10-09}
(I_\mathcal{A}+(t_2^{q-1}-t_1^{q-1},\ldots,t_s^{q-1}-t_1^{q-1})\colon
(t_1\cdots 
t_s)^\infty)\subset I(X)
\end{equation}
with equality if and only if the map $\phi$ is injective. 
\end{theorem}

\begin{proof} We set
$Q=I_\mathcal{A}+(t_2^{q-1}-t_1^{q-1},\ldots,t_s^{q-1}-t_1^{q-1})$.
From the discussion above we have the inclusion $Q\subset I(X)$. By
Theorem~\ref{ipn-ufpe-cinvestav}(b) each variable $t_i$ is not a zero
divisor of $S/I(X)$. It follows readily that 
$(Q\colon(t_1\cdots t_s)^\infty)\subset I(X)$. 

To prove the second part of the theorem we first need to identify the left hand side 
of Eq.~(\ref{happy-sept10-09}) with a lattice ideal for some specific
lattice. Let $A$
be the matrix 
with column vectors $v_1,\ldots,v_s$ and 
consider the lattice 
$$
\mathcal{L}={\rm
ker}_\mathbb{Z}(A)+\mathbb{Z}\{(q-1)(e_i-e_1)\}_{i=2}^s
\subset\mathbb{Z}^s,
$$
where ${\rm ker}_\mathbb{Z}(A)=\{x\in\mathbb{Z}^s\vert\, Ax=0\}$ and 
 $e_i$ denotes the $i${\it th} unit vector of $\mathbb{R}^s$. 
It is seen that 
\begin{equation}\label{saturation-is-lattice}
I(\mathcal{L})=(Q\colon(t_1\cdots t_s)^\infty),
\end{equation} 
see \cite[Corollary~2.5]{EisStu} or \cite[Lemma 7.6]{cca}. This
equality is valid over any field $K$.

$\Rightarrow$) Assume that equality holds in
Eq.~(\ref{happy-sept10-09}). Let $\overline{b}=\overline{(b_i)}$ be 
an element of ${\rm ker}(\phi)$. Then we can write
\begin{equation}\label{sept10-09}
(q-1)b=\sum_{i=1}^sa_iv_i\ \mbox{ with }\  \sum_{i=1}^sa_i=0.
\end{equation}
Consider the homogeneous binomial $f=t^{a^+}-t^{a^-}$, where
$a=(a_i)=a^+-a^-$. From Eq.~(\ref{sept10-09}) we get the equality  
$$ 
x_i^{a_1^+v_{1i}+\cdots+a_s^+v_{si}}=x_i^{a_1^-v_{1i}+\cdots+a_s^-v_{si}}
\ \mbox{ for any }\ x_i\in K^*.
$$
Consequently $f(x^{v_1},\ldots,x^{v_s})=0$ for any sequence
$x_1,\ldots,x_n$ in $K^*$. Then $f$ vanishes on $X$ and is
homogeneous, i.e., $f\in I(X)$. By hypothesis and using
Eq.~(\ref{saturation-is-lattice}), we obtain the equality  $I(X)=I(\mathcal{L})$. Thus
$f=t^{a^+}-t^{a^-}$ belongs to $I(\mathcal{L})$. It is seen that 
$a=a^+-a^-$ belongs to $\mathcal{L}$. Then we can write $a=k+c$, where 
$k\in{\rm ker}_\mathbb{Z}(A)$ and
$c\in\mathbb{Z}\{(q-1)(e_i-e_1)\}_{i=2}^s$. Then from
Eq.~(\ref{sept10-09}) it follows readily that 
$$(q-1)b=Aa=Ak+Ac=Ac=(q-1)Ac',$$ 
for some
$c'\in\mathbb{Z}\{(e_i-e_1)\}_{i=2}^s$. Hence $b=Ac'$, i.e., $b$
belongs to $L$. This means that $\overline{b}=0$ and we have shown
that $\phi$ is injective, as required.  

$\Leftarrow$) Assume that $\phi$ is injective. We now prove the inclusion
$(Q\colon(t_1\cdots t_s)^\infty)\supset I(X)$. Take a binomial 
$f=t^a-t^b$ in $I(X)$ with $a=(a_i)$ and $b=(b_i)$ in $\mathbb{N}^s$.
By Theorem~\ref{ipn-ufpe-cinvestav}(a) it
suffices to prove that $f$ is in $(Q\colon(t_1\cdots t_s)^\infty)$. 
Thus by Eq.~(\ref{saturation-is-lattice}) we need only show that
$a-b\in\mathcal{L}$. We set $v_i=(v_{i1},\ldots,v_{in})$ for
$i=1,\ldots,s$. Since $f$ vanishes on 
$X$ we get 
$$
[x_1^{v_{11}}\cdots x_n^{v_{1n}}]^{a_1}\cdots [x_1^{v_{s1}}\cdots
x_n^{v_{sn}}]^{a_s}=[x_1^{v_{11}}\cdots x_n^{v_{1n}}]^{b_1}\cdots
[x_1^{v_{s1}}\cdots
x_n^{v_{sn}}]^{b_s}\ \mbox{ for all } \ x_i\in K^*.
$$
Let $\beta$ be a generator of the cyclic group $(K^*,\,\cdot\, )$. 
Then for any $(\ell_1,\ldots,\ell_n)$ in
$[1,q-1]^n\cap\mathbb{N}^n$ we can substitute  $x_i=\beta^{\ell_i}$ 
for $i=1,\ldots,n$ in the equality above to obtain
\begin{eqnarray*}
& &[(\beta^{\ell_1})^{v_{11}}\cdots
(\beta^{\ell_n})^{v_{1n}}]^{a_1}\cdots [(\beta^{\ell_1})^{v_{s1}}\cdots 
(\beta^{\ell_n})^{v_{sn}}]^{a_s}=\\
& &\ \ \ \ \ \ \ [(\beta^{\ell_1})^{v_{11}}\cdots
(\beta^{\ell_n})^{v_{1n}}]^{b_1}\cdots 
[(\beta^{\ell_1})^{v_{s1}}\cdots
(\beta^{\ell_n})^{v_{sn}}]^{b_s}\ \mbox{ for all } 
\ 1\leq \ell_i\leq q-1,\, \ell_i\in\mathbb{N}.
\end{eqnarray*}
Therefore for any $\ell=(\ell_1,\ldots,\ell_n)\in
[1,q-1]^n\cap\mathbb{N}^n$ we get
$$
\beta^{a_1\langle\ell,v_1\rangle}\cdots
\beta^{a_s\langle\ell,v_s\rangle}=\beta^{b_1\langle\ell,v_1\rangle}\cdots
\beta^{b_s\langle\ell,v_s\rangle}.
$$
Since $\beta$ has order $q-1$ we obtain
$$
a_1\langle\ell,v_1\rangle+\cdots+a_s\langle\ell,v_s\rangle\equiv 
b_1\langle\ell,v_1\rangle+\cdots+b_s\langle\ell,v_s\rangle\mod(q-1).
$$
If we set $c_i=a_i-b_i$ for all $i$ and 
$\delta=(\delta_i):=c_1v_1+\cdots+c_sv_s$, then 
\begin{equation}\label{sept4-09}
\langle\ell,\delta\rangle
\equiv 0\mod(q-1)
\end{equation}
for any $\ell$ in $[1,q-1]^n\cap\mathbb{N}^n$. Making 
$\ell=(q-1,1\ldots,1)$ and $\ell'=(q-2,1,\ldots,1)$ in
Eq.~(\ref{sept4-09}) we get the equalities
\begin{eqnarray*}
\langle\ell,\delta
\rangle&=&(q-1)\delta_1+\delta_2+\cdots+\delta_n\equiv 0\mod (q-1),\\
\langle\ell',\delta
\rangle&=&(q-2)\delta_1+\delta_2+\cdots+\delta_n\equiv 0\mod (q-1).
\end{eqnarray*}
Consequently, subtracting these equalities, we get that $\delta_1\equiv
0\mod(q-1)$. By an appropriate choice of $\ell$ and $\ell'$ a similar argument
shows that $\delta_i\equiv 0\mod(q-1)$ for all $i$. Therefore we can
write $\delta=(q-1)\gamma$ for some $\gamma\in\mathbb{Z}^n$. Notice
that $\delta\in L$ because $t^a-t^b$ is homogeneous, i.e., because 
$\sum_{i}c_i=0$. Since the
map $\phi$ is injective we obtain that
$\gamma\in L\subset\mathbb{Z}\mathcal{A}$. Hence we can write
$$
\delta=c_1v_1+\cdots+c_sv_s=(q-1)(d_1v_1+\cdots+d_sv_s)
$$ 
for some $d_i$'s in $\mathbb{Z}$. Setting $c=(c_i)$ and $d=(d_i)$, the 
vector $k=(k_i)=c-(q-1)d$ is in ${\rm ker}_\mathbb{Z}(A)$. Notice 
that $\sum_ik_i=0$, because $\sum_ik_iv_i=0$ and $\mathcal{A}$ is
homogeneous. Since 
$\sum_ic_i=0$, by Lemma~\ref{sept4-09-1} we get that $c$ and $k$ are in
$\mathbb{Z}\{e_i-e_1\}_{i=2}^s$. From the equality 
$k=c-(q-1)d$ we obtain that
$(q-1)d\in\mathbb{Z}\{e_i-e_1\}_{i=2}^s$ and since the quotient
group
$$
\mathbb{Z}^s/\mathbb{Z}\{e_i-e_1\}_{i=2}^s
$$
is torsion-free we get that
$d\in\mathbb{Z}\{e_i-e_1\}_{i=2}^s$. Altogether we conclude
that $c=k+(q-1)d$, where $k\in{\rm ker}_\mathbb{Z}(A)$ and $
(q-1)d\in \mathbb{Z}\{(q-1)(e_i-e_1)\}_{i=2}^s$, that is,
$c\in\mathcal{L}$, as required. 
\end{proof}

\begin{remark} If equality occurs in 
Eq.~(\ref{happy-sept10-09}), then $X$ is the projective variety
defined by the binomial ideal
$I_\mathcal{A}+(\{t_i^{q-1}-t_1^{q-1}\}_{i=2}^s)$. This  
will follow from Lemma~\ref{sept25-09} and the proof of
Proposition~\ref{sept29-09-5}.
\end{remark}

\begin{remark}\label{oct15-09} The map $\phi$ is injective if and
only if $q-1$ is not 
a zero divisor of $\mathbb{Z}^n/L$ if and only if the equality 
$(L\colon_{\mathbb{Z}^n}(q-1))=L$ holds, where the left hand side of
the equality is a
colon ideal consisting of all $a\in\mathbb{Z}^n$ such that $(q-1)a\in
L$.
\end{remark}

\begin{corollary}{\bf \cite[Theorem~1]{GRH}}\label{sept15-09}
Let $\mathbb{T}^*=(K^*)^s$ be an affine algebraic torus and let
$\mathbb{T}$
be its image in $\mathbb{P}^{s-1}$ under the map $\varphi_s$. Then
$I(\mathbb{T})=(\{t_i^{q-1}-t_1^{q-1}\}_{i=2}^s)$.  
\end{corollary}

\begin{proof} The set $\mathbb{T}$ is an algebraic toric set
parameterized  by the 
monomials $y^{v_1},\ldots,y^{v_s}$, where $v_i=e_i$ for all $i$. Since 
$I_\mathcal{A}=(0)$ and the group 
$\mathbb{Z}^s/L=\mathbb{Z}^s/\mathbb{Z}\{e_i-e_1\}_{i=2}^s$ is
torsion-free, the equality follows from Theorem~\ref{vila-dictaminadora}. 
\end{proof}

In \cite{GRH} the evaluation codes associated to $\mathbb{T}$ are 
called {\it generalized Reed-Solomon} codes. Thus parameterized  codes are a natural
extension of this sort of codes.

If $D$ is an integral domain and $M$ is a $D$-module, then 
the {\it torsion sub-module\/} of $M$, denoted by $T_D(M)$, is the set of
all $m$ in $M$ such that $pm=0$
for some $0\neq p\in D$. We say that $M$ is 
{\it torsion-free\/} if $T_D(M)=(0)$. In what follows $D$ will always
be the ring 
of integers. Thus, we denote the set of zero divisors and 
the torsion sub-module of $M$ simply by $\mathcal{Z}(M)$ and $T(M)$ 
respectively. 

\begin{lemma}\label{mar18-01} Let
$\mathcal{A}=\{v_1,\ldots,v_s\}\subset\mathbb{Z}^n$, let 
$L=\mathbb{Z}\{v_i-v_1\}_{i=2}^s$ and let
$\mathcal{B}=\{(v_i,1)\}_{i=1}^s$. Then
\begin{enumerate}
\item[($\mathrm{i}_1$)] there is an isomorphism of groups 
$\tau{\colon}T(\mathbb{Z}^n/L)\rightarrow 
T(\mathbb{Z}^{n+1}/\mathbb{Z}\mathcal{B})$, 
given by $\tau(\overline{a})=\overline{(a,0)}$, 
\item[($\mathrm{i}_2$)] $\mathcal{Z}(\mathbb{Z}^n/L)=
\mathcal{Z}(\mathbb{Z}^{n+1}/\mathbb{Z}\mathcal{B})$,
\item[($\mathrm{i}_3$)] if $\mathcal{A}$ is homogeneous, then 
$I_\mathcal{A}=I_\mathcal{B}$.
\end{enumerate}
\end{lemma}

\begin{proof} ($\mathrm{i}_1$): The map $\tau$ is clearly a well 
defined one to one homomorphism of groups. To prove that $\tau$ is onto 
let $\overline{(a,b)}\in T(\mathbb{Z}^{n+1}/\mathbb{Z}\mathcal{B})$
with $a\in\mathbb{Z}^n$, $b\in\mathbb{Z}$. There is $0\neq p\in\mathbb{N}$ such that 
$$
p(a,b)=\lambda_1(v_1,1)+\cdots+
\lambda_s(v_s,1)\ \ \ \ \ \ (\lambda_i\in\mathbb{Z}).
$$ 
Then $pa=\lambda_1v_1+\cdots+\lambda_sv_s$ and
$pb=\lambda_1+\cdots+\lambda_s$.  
Hence we obtain the equality
$$p(a-bv_1)=
\lambda_2(v_2-v_1)+\cdots+
\lambda_s(v_s-v_1).
$$
This means that $\overline{a-bv_1}$ is an element of
$T(\mathbb{Z}^n/L)$. It follows readily that $\tau(\overline{a-bv_1})=
\overline{(a,b)}$. Thus $\tau$ is onto. 
($\mathrm{i}_2$): This is not hard to prove. It follows using that 
the map $\tau$ is an isomorphism. ($\mathrm{i}_3$):
This follows by a direct  application of
\cite[Corollary~7.2.42]{monalg}.
\end{proof}

Using this lemma we will prove the next generalized version of
Theorem~\ref{vila-dictaminadora}, valid for any $\mathcal{A}$. 
The trick to show the next result is to lift 
$\mathcal{A}$ to a homogeneous set $\mathcal{B}$ in
$\mathbb{Z}^{n+1}$.

\begin{corollary}\label{sunday-morning-06-09-09} 
Let $\mathcal{A}=\{v_1,\ldots,v_s\}\subset\mathbb{Z}^n$ and let
$\mathcal{B}=\{(v_1,1),\ldots,(v_s,1)\}$. Then 
\begin{enumerate}
\item[\rm (a)] $
(I_\mathcal{B}+(t_2^{q-1}-t_1^{q-1},\ldots,t_s^{q-1}-t_1^{q-1})\colon
(t_1\cdots t_s)^\infty)\subset I(X)$.
\item[\rm (b)] Equality in {\rm (a)} holds if and only if 
$q-1\notin\mathcal{Z}(\mathbb{Z}^{n+1}/\mathbb{Z}\mathcal{B})$.
\item[\rm (c)] Let $p_1,\ldots,p_m$ be the prime numbers $($if any$)$ that
occur in the factorizations of the invariant factors of
the $\mathbb{Z}$-module $\mathbb{Z}^{n+1}/\mathbb{Z}\mathcal{B}$.
Equality in {\rm (a)} holds 
if and only if either $\mathbb{Z}^{n+1}/\mathbb{Z}\mathcal{B}$ is
torsion-free or $q\not\equiv 1\mod p_i$ for all $i$. 
\end{enumerate}
\end{corollary}
\begin{proof} Let $w$ be a new parameter and let $X^w$ be the image
under the map $\varphi_s$ of the set 
$$
(X^*)^w=\{(x_1^{v_{11}}\cdots x_n^{v_{1n}}w,\ldots,x_1^{v_{s1}}\cdots
x_n^{v_{sn}}w)\vert\, x_i\in K^*\mbox{ for all
}i,\, w\in K^*\}.
$$
Clearly $\mathcal{B}$ is homogeneous because if we set $x_0=e_{n+1}$,
we get $\langle x_0,(v_i,1)\rangle=1$ for all $i$. By
Lemma~\ref{mar18-01} we have $\mathcal{Z}(\mathbb{Z}^n/L)=
\mathcal{Z}(\mathbb{Z}^{n+1}/\mathbb{Z}\mathcal{B})$, where 
$L=\mathbb{Z}\{v_i-v_1\}_{i=2}^s$. Therefore (a) and (b)
follow at once from Theorem~\ref{vila-dictaminadora} and
Remark~\ref{oct15-09} because $X=X^w$.  

We now prove (c). If $\mathbb{Z}^{n+1}/\mathbb{Z}\mathcal{B}$ is
torsion-free, then
equality holds in (a) by part (b). Hence we may assume that this
module has torsion. By the  
fundamental structure theorem of finitely 
generated abelian groups (see \cite[pp.~187-188]{JacI}) we have
\begin{equation}\label{fundamental-st-theo}
\mathbb{Z}^{n+1}/\mathbb{Z}\mathcal{B}\simeq \mathbb{Z}^{r_0}\times 
\mathbb{Z}_{q_1^{\alpha_1}}\times\cdots\times\mathbb{Z}_{q_r^{\alpha_r}},
\end{equation}
where $q_i\in\{p_1,\ldots,p_m\}$ and
$r_0=n+1-{\rm rank}(\mathbb{Z}\mathcal{B})$. From
Eq.~(\ref{fundamental-st-theo}) it is seen that one has the 
equality $\mathcal{Z}(\mathbb{Z}^{n+1}/\mathbb{Z}\mathcal{B})=
\cup_{i=1}^m(p_i)$. Therefore, by (b), equality holds in (a) if and only
if $q-1\notin \cup_{i=1}^m(p_i)$ if and only if $q\not\equiv 1\mod p_i$
for all $i$.
\end{proof}

\begin{corollary}\label{sept29-09-2} Let $G$ be a simple graph with vertex set
$V_G=\{y_1,\ldots,y_n\}$, edge set $E_G$, and let 
$\mathcal{A}$ be the set of all $e_i+e_j$ such that 
$\{y_i,y_j\}\in E_G$. If $c_1$ is the number of non-bipartite
connected components of $G$, then  the equality 
\begin{equation}\label{oct16-09}
(I_\mathcal{A}+(t_2^{q-1}-t_1^{q-1},\ldots,t_s^{q-1}-t_1^{q-1})\colon
(t_1\cdots t_s)^\infty)=I(X)
\end{equation}
holds if and only if either $0\leq c_1\leq 1$ or $c_1\geq 2$ and
${\rm char}(K)=2$. In particular equality holds for any finite field
$K$ if $G$ is connected or if $G$ is bipartite.
\end{corollary}

\begin{proof} Let $\mathcal{A}=\{v_1,\ldots,v_s\}$ and let
$\mathcal{B}=\{(v_1,1),\ldots,(v_s,1)\}$ be a lifting of
$\mathcal{A}$. 
Notice that
$I_\mathcal{A}=I_\mathcal{B}$ because $\mathcal{A}$ is homogeneous,
see Lemma~\ref{mar18-01}($\mathrm{i}_3$). We denote the
matrix whose columns are the vectors in $\mathcal{A}$ (resp.
$\mathcal{B}$) by $A$ (resp.
$B)$. The matrices $A$ and $B$ have the same rank $r$. We denote the 
greatest common divisor of all the non-zero $r\times r$
sub-determinants of $A$ (resp. $B)$ by $\Delta_r(A)$ (resp.
$\Delta_r(B)$).  

We claim that $\Delta_r(B)=2^{c_1-1}$ if $c_1\geq 1$ and
$\Delta_r(B)=1$ if $c_1=0$. If $c_1=0$, then $G$ is bipartite. Thus 
$\Delta_r(B)=1$ because in this case $A$ is totally unimodular 
\cite[p.~273]{Schr}, i.e., any sub-determinant of $A$ is equal to $0$
or $\pm 1$. Assume that $c_1\geq 1$, i.e., $G$ is not bipartite.
There is an exact sequence of groups
\begin{equation}\label{exact-sequence-birational}
0\longrightarrow
T(\mathbb{Z}^{n+1}/\mathbb{Z}{\mathcal B})
\stackrel{\vartheta}{\longrightarrow}
T(\mathbb{Z}^n/\mathbb{Z}{\mathcal A})\stackrel{\psi}{\longrightarrow} 
\mathbb{Z}_2\longrightarrow 0, 
\end{equation}
where the homomorphisms are defined as follows. For
$a=(a_1,\ldots,a_n)\in\mathbb{Z}^n$ 
and $b\in\mathbb{Z}$, we set  
$$
\vartheta(\overline{a,b)}=\overline{a}\ \mbox{ and }\ 
\psi(\overline{a})=\overline{a_1+\cdots+a_n}.
$$
It is not hard to verify that $\vartheta$ is injective, $\psi$ is
onto, and ${\rm im}(\vartheta)={\rm ker}(\psi)$. The exact sequence of 
Eq.~(\ref{exact-sequence-birational}) is a particular case 
of \cite[Eq.~$(*)$, p.~2044]{birational}. It is well known
\cite[pp.~187-188]{JacI} that the orders of the groups $T(\mathbb{Z}^{n}/\mathbb{Z}\mathcal{A})$ and 
$T(\mathbb{Z}^{n+1}/\mathbb{Z}\mathcal{B})$ are $\Delta_r(A)$ and $\Delta_r(B)$ respectively.  
Therefore, using the exact sequence above, 
we get $\Delta_r(A)=2\Delta_r(B)$. By a result of \cite{Kulk1} we have
\begin{equation}\label{apr6-01-1}
\mathbb{Z}^n/\mathbb{Z}{\mathcal A}
\simeq \mathbb{Z}^{n-r}\times
\mathbb{Z}_2^{c_1}=\mathbb{Z}^{c_0}\times\mathbb{Z}_2^{c_1}
\end{equation}
and $r=n-c_0$, where $c_0$ is the number of bipartite components of
$G$. Hence $\Delta_r(A)=2^{c_1}$, and consequently
$\Delta_r(B)=2^{c_1-1}$ as claimed. This means that
$\mathbb{Z}^{n+1}/\mathbb{Z}\mathcal{B}$ is torsion-free if and
only if $c_1=1$. It also means that $p_1=2$ is the only prime
factor that can occur in the factorizations of the invariant factors of 
$\mathbb{Z}^{n+1}/\mathbb{Z}\mathcal{B}$. The number of elements
of $K$ is equal to $q=p^u$ for some prime number $p$ and some $u\geq
1$, where $p$ is the characteristic of the field $K$. Altogether, by 
Corollary~\ref{sunday-morning-06-09-09}(c), we get that equality holds
in Eq.~(\ref{oct16-09}) if and only if $0\leq c_1\leq 1$ or $c_1\geq 2$ and 
$p^u\not\equiv 1\mod 2$ if and only if $0\leq c_1\leq 1$ or $c_1\geq 2$ 
and $p=2$. 
\end{proof}

\begin{example}\rm Let $\mathcal{A}$ be the point configuration consisting of the
following  points in $\mathbb{Z}^6$:
$$
\begin{array}{ccc}
v_1=(1,1,0,0,0,0),&v_2=(0,1,1,0,0,0),&v_3=(1,0,1,0,0,0),\\
v_4=(0,0,0,1,1,0),&v_5=(0,0,0,0,1,1),&v_6=(0,0,0,1,0,1).
\end{array}
$$
In this case we have
$\mathbb{Z}^6/\mathbb{Z}\mathcal{A}\simeq\mathbb{Z}_2\times\mathbb{Z}_2$
and
$\mathbb{Z}^7/\mathbb{Z}\mathcal{B}\simeq\mathbb{Z}\times\mathbb{Z}_2$.
If $K$ is a finite field with $q=2^m$ elements, then $q\not\equiv
1\mod 2$ and $I_\mathcal{A}=I_\mathcal{B}=0$. Thus using
Corollary~\ref{sunday-morning-06-09-09}(c) we get the equality
$I(X)=(\{t_i^{q-1}-t_1^{q-1}\}_{i=2}^6)$. If $K$ 
is a field with $3$ elements, then using {\em
Macaulay\/}$2$ \cite{mac2} together with
Theorem~\ref{ipn-ufpe-cinvestav} it is seen that $I(X)$ is minimally 
generated by $15$ binomials. In this case we do not have  
equality in Corollary~\ref{sunday-morning-06-09-09}(a). 
\end{example}

The next result can be shown using the argument in the proof 
of Theorem~\ref{ipn-ufpe-cinvestav}.

\begin{theorem}\label{ipn-ufpe-cinvestav-1} Let
$B=K[t_1,\ldots,t_s,y_0,y_1,\ldots,y_n,z]$ be  
a polynomial ring over a finite field $K=\mathbb{F}_q$ and let
$v_i\in\mathbb{Z}^n$ for all
$i$. The following holds{\rm :}
\begin{enumerate}
\item[\rm (a)]
$I(X)=(y^{v_1^-}t_1-y^{v_1^+}z,\ldots,y^{v_s^-}t_s-y^{v_s^+}z,
y_1^{q-1}-1,\ldots,y_n^{q-1}-1,y_0y_1\cdots y_n-1)\cap S$. 
\item[\rm (b)] $I(X)$ is a Cohen-Macaulay lattice ideal and $\dim S/I(X)=1$.
\end{enumerate}
\end{theorem}

\section{The length of parameterized codes and the degree of $S/I(X)$}
\label{degree-of-coordinate-ring}  
We continue
using the 
definitions and terms from the introduction and 
from Section~\ref{section-parameterized}. Let
$\mathcal{A}=\{v_1,\ldots,v_s\}\subset\mathbb{Z}^n$ and let $X$ be an
algebraic 
toric set parameterized by the Laurent monomials
$y^{v_1},\ldots,y^{v_s}$. In this section we study
$|X|$, the degree of $S/I(X)$. The motivation to study $|X|$ comes
from coding theory because this number represents the length of
$C_X(d)$, the parameterized code of order $d$.

As before, we denote the Hilbert polynomial of $S/I(X)$ by $h_X(t)$. 
The {\it index of regularity\/} of $S/I(X)$, denoted by 
${\rm reg}(S/I(X))$, is the least integer $p\geq 0$ such that
$h_X(d)=H_X(d)$ for $d\geq p$. The degree and the regularity index can be
read off the Hilbert series as we now explain. The Hilbert series of
$S/I(X)$ can be 
written as
$$
F_X(t):=\sum_{i=0}^{\infty}H_X(i)t^i=\sum_{i=0}^{\infty}\dim_K(S/I(X))_it^i=
\frac{h_0+h_1t+\cdots+h_rt^r}{1-t},
$$
where $h_0,\ldots,h_r$ are positive integers. Indeed
$h_i=\dim_K(S/(I(X),t_s))_i$.  
This follows from the fact that $I(X)$ is a Cohen-Macaulay lattice
ideal. The number
$r$  equals the regularity index of $S/I(X)$ and the degree of
$S/I(X)$ equals $h_0+\cdots+h_r$ (see \cite{Sta1} or 
\cite[Corollary~4.1.12]{monalg}).  

Although Theorems~\ref{ipn-ufpe-cinvestav}
and \ref{ipn-ufpe-cinvestav-1} provide an effective method to compute
the degree with {\em Macaulay\/}$2$ \cite{mac2}, we seek other methods that can
lead to explicit formulas for $|X|$ for certain families of point
configurations, especially for these arising from finite graphs. 

At the other end, the number of elements of $X^*$, the affine counterpart of $X$,
can alternatively be obtained by using linear algebra methods over the ring
$\mathbb{Z}/(q-1)\mathbb{Z}$, i.e., by solving linear systems over this
ring. This may then be used to estimate $|X|$. As mentioned before, some of the results of this
paper have an affine version. We can think of this linear algebra
approach to compute $|X^*|$ 
as the analog of
Proposition~\ref{sept29-09}, which is a device
that enables  to use linear programming methods.
The multiplicity of approaches is a hint of the mathematical richness embodied in the  
parametrization models dealt with
in this work.

We begin by presenting a direct method, 
based on integer programming \cite{Schr}, to compute the degree of
$S/I(X)$. A key element here is the fact that $X$ is a multiplicative
group as explained in Section~\ref{section-parameterized}. 
Let $\mathbb{T}^*=(K^*)^n$ be an affine algebraic torus of
dimension $n$. There is a surjective homomorphism of multiplicative groups
$$
\theta\colon\mathbb{T}^*\longrightarrow X\, ;\ \ \ \ \ \ \ \ 
(x_1,\ldots,x_n)\stackrel{\theta}{\longmapsto}
[(x^{v_1},\ldots,x^{v_s})]. 
$$
Therefore $\mathbb{T}^*/{\rm ker}(\theta)\simeq X$ and
$|\mathbb{T}^*|=(q-1)^{n}=|X||{\rm ker}(\theta)|$. Thus computing
$|X|$ amounts to computing $|{\rm ker}(\theta)|$. 

\begin{lemma}\label{oct16-09-1}
Let $(x_i)=(\beta^{\ell_1},\ldots,\beta^{\ell_n})\in\mathbb{T}^*$ with
$\beta$ a generator of 
$(K^*,\,\cdot\, )$ and $0\leq \ell_i\leq q-2$ for all $i$. Then
$(x_i)\in{\rm ker}(\theta)$ if and only if there are unique integers
$\lambda_1,\ldots,\lambda_s,\mu$ such that 
$$
{\ell}A=(q-1)\lambda+\mu\mathbf{1};\ \ 0\leq\mu\leq q-2;\
\ell=(\ell_i);\ \lambda=(\lambda_i);\ \mathbf{1}=(1,\ldots,1).
$$
\end{lemma}

\begin{proof} Assume that $(x_i)\in{\rm ker}(\theta)$. Then 
$[(x^{v_1},\ldots,x^{v_s})]=[\mathbf{1}]$. This means that there is an
integer $\mu$ such that $0\leq\mu\leq q-2$ and 
$$
\beta^{\langle v_i,\ell\rangle}=\beta^\mu\mbox{ for all }i.
$$
Hence there are integers $\lambda_1,\ldots,\lambda_s$ such that 
$$
{\langle v_i,\ell\rangle}-\mu=(q-1)\lambda_i\mbox{ for all }i\ 
\Rightarrow\ {\ell}A=(q-1)\lambda+\mu\mathbf{1},
$$
as required. To show the uniqueness assume that 
${\langle v_i,\ell\rangle}-\mu=(q-1)\lambda_i$ and 
${\langle v_i,\ell\rangle}-\mu'=(q-1)\lambda_i'$ for some $i$. Then
$(q-1)(\lambda_i-\lambda_i')=\mu'-\mu$. Since $|\mu'-\mu|$ is at most
$q-2$, we get $\lambda_i=\lambda_i'$ and $\mu'=\mu$. The converse
follows readily by direct substitution of $x_i=\beta^{\ell_i}$ 
into $[(x^{v_1},\ldots,x^{v_s})]$.
\end{proof}

\begin{remark}
If $v_i\in\mathbb{N}^n$, then $\lambda_i\geq 0$. This
follows by dividing $\langle v_i,\ell\rangle$ by $(q-1)$.
\end{remark}

\begin{proposition}\label{sept29-09} The map $\beta^\ell\mapsto (\ell,\lambda,\mu)$ gives a
bijection between ${\rm ker}(\theta)$ and the integral vectors 
of the polytope
$$
\mathcal{P}=\{(\ell,\lambda,\mu)\vert\, \ell=(\ell_i);\,
\lambda=(\lambda_i);\, 
{\ell}A=(q-1)\lambda+\mu\mathbf{1};\, 0\leq \ell_i\leq q-2\mbox{ for
all }i;0\leq\mu\leq q-2\}.
$$
In particular the number of integral vectors of $\mathcal{P}$
equals $|{\rm ker}(\theta)|$.
\end{proposition}

\begin{proof} By Lemma~\ref{oct16-09-1} the map $\beta^\ell\mapsto
(\ell,\lambda,\mu)$ is well defined and bijective.
\end{proof}

\begin{example}\rm Let $A$ be the matrix with column vectors  
$v_1=(1,1,0,0)$, $v_2=(0,1,1,0)$, $v_3=(0,0,1,1)$, $v_4=(1,0,0,1)$.
Let $K$ be a field with $q=5$ elements. The integral points of
$\mathcal{P}$ and the elements of ${\rm ker}(\theta)$ 
can be found directly using 
{\it Porta\/} \cite{porta}. A computation with this program
shows that $\mathcal{P}\cap\mathbb{Z}^{n+s+1}$ has $16$ points and
that ${\ker}(\theta)$ is equal to 
$$
\begin{array}{cccc}
(\beta^0, \beta^0, \beta^0, \beta^0  ),&
(\beta^0, \beta^1, \beta^0, \beta^1  ),&
(\beta^0, \beta^2, \beta^0, \beta^2  ),&
(\beta^0, \beta^3, \beta^0, \beta^3  ),\\
(\beta^1, \beta^0, \beta^1, \beta^0  ),&
(\beta^1, \beta^1, \beta^1, \beta^1  ),&
(\beta^1, \beta^2, \beta^1, \beta^2  ),&
(\beta^1, \beta^3, \beta^1, \beta^3  ),\\
(\beta^2, \beta^0, \beta^2, \beta^0  ),&
(\beta^2, \beta^1, \beta^2, \beta^1  ),&
(\beta^2, \beta^2, \beta^2, \beta^2  ),&
(\beta^2, \beta^3, \beta^2, \beta^3  ),\\
(\beta^3, \beta^0, \beta^3, \beta^0  ),&
(\beta^3, \beta^1, \beta^3, \beta^1  ),&
(\beta^3, \beta^2, \beta^3, \beta^2  ),&
(\beta^3, \beta^3, \beta^3, \beta^3  ).
\end{array}
$$
Hence in this case one has $4^{4}=(q-1)^{n}=|X||{\rm
ker}(\theta)|=|X|16$. Then $|X|=16$.
\end{example}

Before we state our next result, recall that a subset 
$\mathcal{B}\subset\mathbb{Z}^{n+1}$ is called a {\it Hilbert basis} if
$\mathbb{N}\mathcal{B}=\mathbb{R}_+\mathcal{B}\cap\mathbb{Z}^{n+1}$,
where $\mathbb{N}\mathcal{B}$ is the semigroup generated by
$\mathcal{B}$, and $\mathbb{R}_+\mathcal{B}$ is the {\it polyhedral
cone\/} generated by $\mathcal{B}$  consisting of the 
linear combinations of $\mathcal{B}$ with non-negative 
coefficients. A polyhedral cone containing no lines is called 
{\it pointed}. The subgroup of $\mathbb{Z}^{n+1}$ 
generated by $\mathcal{B}$ is denoted by $\mathbb{Z}\mathcal{B}$. The
ideal $I(X)$ is called a {\it complete 
intersection\/} if it can be generated by $s-1$ homogeneous
polynomials of $S$.

\begin{theorem}\label{sept23-09} 
Let $\mathcal{B}=\{(v_1,1),\ldots,(v_s,1)\}$ and let $r={\rm
rank}(\mathbb{Z}\mathcal{B})$. If the polyhedral cone
$\mathbb{R}_+\mathcal{B}$ is pointed and $\mathcal{B}$ is a 
Hilbert basis, then $(q-1)^{r-1}$ divides $|X|$.
\end{theorem}

\begin{proof} By \cite{gerards-sebo}, after permutation
of the $(v_i,1)$'s,  we may assume that 
$\mathcal{B}'=\{(v_1,1),\ldots,(v_r,1)\}$ is a Hilbert basis and a
linearly independent set. It is a fact that $\mathcal{B}$ is a Hilbert basis
if and only if
$\mathbb{R}_+\mathcal{B}\cap\mathbb{Z}\mathcal{B}=\mathbb{N}\mathcal{B}$
and $\mathbb{Z}^{n+1}/\mathbb{Z}\mathcal{B}$ is a torsion-free 
group. This fact can be shown using lattice theory. In
Lemma~\ref{dec3-08} we show the part of this fact that we
really need, namely that $\mathcal{B}'$
is a Hilbert 
basis if and only if the group
$\mathbb{Z}^{n+1}/\mathbb{Z}\mathcal{B}'$ is torsion-free.

Consider the algebraic toric set parameterized by $y^{v_1},\ldots,
y^{v_r}$:
$$X_1=\{[(x^{v_1},\ldots,x^{v_r})]\vert\, x_i\in
K^*\mbox{ for all }i\}\subset\mathbb{P}^{r-1}.
$$ 
Since $I_{\mathcal{B}'}=(0)$ and
$\mathbb{Z}^{n+1}/\mathbb{Z}\mathcal{B}'$ is torsion-free, 
by Corollary~\ref{sunday-morning-06-09-09}(b)  we obtain the equality
$$
I(X_1)=(\{t_i^{q-1}-t_1^{q-1}\}_{i=2}^r).
$$
Thus $I(X_1)$ is a complete intersection generated by $r-1$ forms of degree
$q-1$. For complete intersections there is an explicit formula for the
Hilbert series \cite[p.~104]{monalg}. Hence using this formula we
get that the degree of $K[t_1,\ldots,t_r]/I(X_1)$ is equal to
$(q-1)^{r-1}$, i.e., $|X_1|=(q-1)^{r-1}$. To complete the proof
consider the epimorphism 
$$
\theta_1\colon\mathbb{T}^*\longrightarrow X_1\, ;\ \ \ \ \ \ \ \ 
(x_1,\ldots,x_n)\stackrel{\theta_1}{\longmapsto}
[(x^{v_1},\ldots,x^{v_r})], 
$$
where $\mathbb{T}^*=(K^*)^n$ is an affine algebraic torus. 
Since ${\rm ker}(\theta)\subset {\rm ker}(\theta_1)$, there is an
epimorphism $\overline{\theta}_1\colon X\rightarrow X_1$ such that the diagram 
\begin{center}
$$
\setlength{\unitlength}{.040cm}
\begin{picture}(20,10)(0,-5)
\put(-50,0){$\mathbb{T}^*$}
\put(7,0){$X_1$}
\put(-35,3){\vector(1,0){30}}
\put(-20,8){$\theta_1$}
\put(-46,-5){\vector(0,-1){20}}
\put(-44,-15){$\theta$}
\put(-50,-37){$X$}
\put(-9,-28){$\overline{\theta}_1$}
\put(-35,-33){\vector(3,2){42}}
\end{picture}
$$
\end{center}
\vspace{1.3cm} 

\noindent is commutative. Therefore $|X_1|=(q-1)^{r-1}$ divides $|X|$.
\end{proof}

\begin{definition} Let $\mathcal{P}\subset\mathbb{R}^n$ be a lattice polytope, i.e.,
$\mathcal{P}$ is the convex hull of a finite set of integral points 
in $\mathbb{R}^n$. The {\it relative volume\/} of $\mathcal{P}$, 
denoted by ${\rm vol}(\mathcal{P})$, is given by 
$$
{\rm vol}(\mathcal{P}):=
\lim_{i\rightarrow\infty}\frac{|\mathbb{Z}^n\cap i\mathcal{P}|}{i^d},
$$
where $d=\dim(\mathcal{P})$, $i\in\mathbb{N}$, 
$i\mathcal{P}=\{ix\vert\, x\in \mathcal{P}\}$. 
\end{definition}

\begin{lemma}\label{dec3-08} Let
$\mathcal{B}'=\{u_1,\ldots,u_r\}\subset\mathbb{Z}^{n+1}$ be a set of
linearly independent vectors. Then $\mathcal{B}'$ is a Hilbert 
basis if and only if $\mathbb{Z}^{n+1}/\mathbb{Z}\mathcal{B}'$ is
torsion-free. 
\end{lemma}

\begin{proof} Let $B'$ be the matrix with column
vectors $u_1,\ldots,u_r$ and let $\Delta_r(B')$ be the greatest
common divisor of all the non-zero $r\times r$ sub-determinants of
$B'$. Assume that $\mathcal{B}'$ is a Hilbert basis. Since
$|T(\mathbb{Z}^{n+1}/\mathbb{Z}\mathcal{B}')|$ is equal to
$\Delta_r(B')$, we need
only show $\Delta_r(B')=1$.  According to \cite[Lemma~2.1]{ehrhart} there are vectors
$\gamma_1,\ldots,\gamma_r$ in $\mathbb{Z}^{n+1}$ such that  
\begin{equation*}
\mathbb{R}\mathcal{B}'\cap\mathbb{Z}^{n+1}=\mathbb{Z}\gamma_1\oplus\cdots\oplus
\mathbb{Z}\gamma_r,
\end{equation*}
where $\mathbb{R}\mathcal{B}'$ is the vector space spanned by
$\mathcal{B}'$. Then we can write 
\begin{equation*}
u_i=c_{i1}\gamma_1+\cdots+c_{ir}\gamma_r\ \ \ (i=1,\ldots,r)
\end{equation*}
where $C=(c_{ij})$ is an integral matrix. By
\cite[Remark~2.2]{ehrhart}, we have   
$$
\Delta_r(B')=r!{\rm vol}({\rm
conv}(0,u_1,\ldots,u_r))=|\det(C)|.
$$
To complete the proof it suffices to show that
$|\det(C)|=1$. Let $c_1,\ldots,c_r$ be the rows of $C$. 
As $\mathcal{B'}$ is a Hilbert basis, it is seen that
the rows of $C$ form a Hilbert basis. Let $\mathcal{Q}=[0,1]^r$ and let $\mathcal{P}$ be the 
parallelotope
$$
\mathcal{P}=\{\lambda_1c_1+\cdots+\lambda_rc_r\vert\, 0\leq \lambda_i\leq 1\}.
$$
Recall that ${\rm vol}(\mathcal{P})=|\det(C)|$. As $c_1,\ldots,c_r$
are linearly independent and form a Hilbert basis, 
we have
$$
(k+1)^r=|k\mathcal{Q}\cap\mathbb{Z}^r|=|k\mathcal{P}\cap\mathbb{Z}^r|
\ \mbox{ for all }\ k\in\mathbb{N}. 
$$
Therefore 
$$
1=\lim_{k\rightarrow\infty}\frac{(k+1)^r}{k^r}=
\lim_{k\rightarrow\infty}\frac{|k\mathcal{Q}\cap\mathbb{Z}^r|}{k^r}=
\lim_{k\rightarrow\infty}\frac{|k\mathcal{P}\cap\mathbb{Z}^r|}{k^r}=
{\rm vol}(\mathcal{P}).
$$
Thus we have shown $1={\rm vol}(\mathcal{P})=|\det(C)|$, as required.
The converse follows readily. 
\end{proof}

\begin{corollary}\label{sept29-09-1} Let $G$ be a connected graph with vertex set
$V_G=\{y_1,\ldots,y_n\}$, edge set $E_G$, and let 
$\mathcal{A}=\{v_1,\ldots,v_s\}$ be the set of all
$e_i+e_j\in\mathbb{R}^n$ such that 
$\{y_i,y_j\}\in E_G$. Then $|X|=(q-1)^{n-1}$ if 
$G$ is non-bipartite and $|X|=(q-1)^{n-2}$ if $G$ is bipartite.
\end{corollary}

\begin{proof} Assume that $G$ is non-bipartite. Then $G$ has a
connected subgraph $H$ with the same vertex set and with a unique
cycle of odd length. We may assume that $\{v_1,\ldots,v_n\}$ is the
set of all $e_i+e_j$ such that $\{y_i,y_j\}$ is an edge of $H$. Let
$B'$ be the matrix whose columns are the vectors in 
$\mathcal{B}'=\{(v_1,1),\ldots,(v_n,1)\}$. Then
$\Delta_n(B')=1$, see the proof of Corollary~\ref{sept29-09-2}. As 
$|T(\mathbb{Z}^{n+1}/\mathbb{Z}\mathcal{B}')|$ equals $\Delta_n(B')$, we
obtain that $\mathbb{Z}^{n+1}/\mathbb{Z}\mathcal{B}'$ is 
torsion-free. Therefore, by Lemma~\ref{dec3-08}, 
the set $\mathcal{B}'$ is a
Hilbert basis and
generates a group of rank $n$. Hence by Theorem~\ref{sept23-09} we get
that $(q-1)^{n-1}$ divides $X_1$, where 
$$X_1=\{[(x^{v_1},\ldots,
x^{v_n})] \vert\, x_i\in K^*\mbox{ for all
}i\}\subset\mathbb{P}^{n-1}.$$ 
There is a
well defined epimorphism 
$$
\overline{\theta}_1\colon X\longrightarrow X_1
\, ;\ \ \ \ \ \ \ \ 
[(x^{v_1},\ldots,x^{v_s})]\stackrel{\overline{\theta}_1}{\longmapsto}
[(x^{v_1},\ldots,x^{v_n})]
$$
induced by the projection map
$[(\alpha_1,\ldots,\alpha_s)]\mapsto[(\alpha_1,\ldots,\alpha_n)]$. 
Thus $|X_1|$
divides $|X|$. Hence $(q-1)^{n-1}$ divides $|X|$. On the
other hand the kernel of the map 
$$
\theta\colon\mathbb{T}^*\rightarrow X\, ;\ \ \ \ \ \ \ \ 
(x_1,\ldots,x_n)\stackrel{\theta}{\longmapsto}
[(x^{v_1},\ldots,x^{v_s})]
$$ 
contains the diagonal 
subgroup 
$\mathcal{D}^*=\{(\lambda,\ldots,\lambda)\vert\, \lambda\in K^*\}$.
Thus $|X|$ divides $(q-1)^{n-1}$. Putting altogether we get
$|X|=(q-1)^{n-1}$. 

Assume that $G$ is bipartite. We may assume that
$V_1=\{y_1,\ldots,y_p\}$, $V_2=\{y_{p+1},\ldots,y_n\}$ is the
bipartition of $G$. The graph $G$ has a spanning tree $H$ with the
same vertex set. We may assume that $\{v_1,\ldots,v_{n-1}\}$ is the
set of all $e_i+e_j$ such that $\{y_i,y_j\}$ is an edge of $H$. Let
$B'$ be the matrix whose columns are the vectors in 
$\mathcal{B}'=\{(v_1,1),\ldots,(v_{n-1},1)\}$. Then
$\Delta_{n-1}(B')=1$, see the proof of Corollary~\ref{sept29-09-2}.
Therefore, by Lemma~\ref{dec3-08}, 
the set $\mathcal{B}'$ is a
Hilbert basis and
generates a group of rank $n-1$. 
Hence by Theorem~\ref{sept23-09} we get
that $(q-1)^{n-2}$ divides $|X_1|$, where 
$$X_1=\{[(x^{v_1},\ldots,
x^{v_{n-1}})] \vert\, x_i\in K^*\mbox{ for all
}i\}\subset\mathbb{P}^{n-2}.$$
There is an
epimorphism $\overline{\theta}_1\colon X\rightarrow X_1$. Thus $|X_1|$
divides $|X|$ and consequently $(q-1)^{n-2}$ divides $|X|$. On the
other hand the kernel of the map 
$\theta\colon\mathbb{T}^*\rightarrow X$ 
contains the set $\Gamma$ of all vectors of the form
$$
\underbrace{(\beta^{a},\ldots,\beta^{a}}_{p-\mbox{\tiny entries}}
\underbrace{\beta^{b},\ldots,\beta^{b})}_{(n-p)-\mbox{\tiny
entries}}
$$
with $0\leq a,b\leq q-2$. Indeed any of these vector maps 
to  $[(\beta^{a+b},\ldots,\beta^{a+b})]=[\mathbf{1}]$ under the 
map $\theta$. Since $|\Gamma|=(q-1)^2$ we obtain that
$|X|\leq(q-1)^{n-2}$. Altogether $|X|=(q-1)^{n-2}$.
\end{proof}

Parameterized  codes arising from complete
bipartite graphs have been studied in \cite{GR}. In loc. cit. one can
find formulas for some of its basic parameters. As an
application we recover a formula for the length of these codes.

\begin{corollary}{\rm \cite[Theorem~5.1]{GR}} 
If $G$ is a complete bipartite graph with $n$ vertices, then 
the length of the parameterized code $C_X(d)$ is equal to $(q-1)^{n-2}$.
\end{corollary}

The hypothesis that $G$ is connected is essential in
Corollary~\ref{sept29-09-1}: 

\begin{example}\label{singleton-bound} 
Let $K=\mathbb{F}_7$ and let $X$ be the algebraic toric set
parameterized by the monomials
$y_1y_2,y_2y_3,y_1y_3,y_4y_5,y_5y_6,y_4y_6$. Using
Theorem~\ref{ipn-ufpe-cinvestav} and {\em Macaulay\/}$2$ \cite{mac2}
we get: 
$$
|X|={\rm degree}\, S/I(X)=(q-1)^{n-1}/2=3888,\ \ {\rm reg}\, S/I(X)=16,
$$
the ideal $I(X)$ is generated by $15$ binomials, and the Hilbert
function of $S/I(X)$ is given by
$$
\begin{array}{lllll}
H_X(0)=1,&H_X(1)=6,&H_X(2)=21,&H_X(3)= 56,&H_X(4)= 126,\\
 H_X(5)= 252,& H_X(6)= 457,&H_X(7) = 762,&H_X(8) = 1182, &H_X(9)= 1712,\\
H_X(10) = 2313,&H_X(11) = 2898,&H_X(12) = 3373,&H_X(13)=
3678,&H_X(14)= 3828,\\
H_X(15)= 3878,&H_X(16)= 3888.& & &
\end{array}
$$
Thus the length of the parameterized  code $C_X(d)$ of order $d$ is $3888$
and its dimension is $H_X(d)$. Then the Singleton bound 
gives that the minimum distance of $C_X(15)$ is at most $11$.
\end{example}

\section{The vanishing ideal of certain projective binomial
varieties}\label{vanishing-ideal-projective}

Let $K=\mathbb{F}_q$ be a finite field and let
$\mathcal{A}=\{v_1,\ldots,v_s\}\subset\mathbb{Z}^n$ be a point
configuration. In this section we study the geometric
structure of $X$, the algebraic toric set parameterized  by
$y^{v_1},\ldots,y^{v_s}$. A sufficient condition is given for $X$ to be
a projective variety defined by binomials and a finite Nullstellensatz is brought up in
this connection. We prove that certain projective binomial varieties
are parameterized by Laurent monomials.

Let $V(I_\mathcal{A})=\{[\alpha]\in\mathbb{P}^{s-1}\vert\,
f(\alpha)=0\mbox{ for all }f\in I_\mathcal{A}\mbox{ with }f\mbox{
homogeneous}\}$ be the projective toric variety defined by
the toric ideal $I_\mathcal{A}$. We shall be 
interested in the following projective binomial variety
$V_\mathcal{A}$ and in its
corresponding ideal $I(V_\mathcal{A})$:
$$
V_{\mathcal A}:=V(I_\mathcal{A})\cap \mathbb{T}=
V(I_\mathcal{A}+(\{t_i^{q-1}-t_1^{q-1}\}_{i=2}^s)),
$$
where
$\mathbb{T}=V(\{t_i^{q-1}-t_1^{q-1}\}_{i=2}^s)=\{[(\alpha_i)]\in\mathbb{P}^{s-1}\vert\,
\alpha_i\in K^*\mbox{ for all }i\}$ is a projective torus. 

First we prove 
that $V_\mathcal{A}$ is parameterized  by Laurent monomials provided
that $\mathcal{A}$ is homogeneous. As in previous sections, let $A$
be the matrix with column vectors 
$v_1,\ldots,v_s$. Recall that ${\rm ker}_\mathbb{Z}(A)$ is 
a free abelian group of finite rank. Let $c_1,\ldots,c_m$ be a set of
generators of ${\rm ker}_\mathbb{Z}(A)$. Write $c_i=(c_{i1},\ldots,c_{is})$ for
$1\leq i\leq m$. Consider the linear system 
\begin{eqnarray}\label{sept26-09}
c_{11}x_1+\cdots+c_{1s}x_s-(q-1)x_{s+1}&=&0\nonumber\\
\vdots\ \ \ \ \  \ \ \ \ \ \ \ \ \ \  \ \ \vdots \ \ \ \ \ \ \  \  \
\ \ \ \ \ \ \vdots \ \ \ \ \ \ \ & &\vdots\\   
c_{m1}x_1+\cdots+c_{ms}x_s-(q-1)x_{s+m}&=&0.\nonumber
\end{eqnarray}
The integral solutions of this system form a free abelian group 
of finite rank. Let 
\begin{eqnarray}\label{sept26-09-1}
\gamma_1&=&(\alpha_{11},\ldots,\alpha_{s1},\alpha_{(s+1)1},\ldots,
\alpha_{(s+m)1})\nonumber\\ 
\vdots&  &\ \ \ \vdots \ \ \ \ \ \ \ \ \ \  \vdots\ \ \ \ 
\ \ \ \ \vdots\ \ \ \ \ \ \ \ \ \ \ \ \ \  \vdots\\
\gamma_k&=&(\alpha_{1k},\ldots,\alpha_{sk},\alpha_{(s+1)k},\ldots,
\alpha_{(s+m)k})\nonumber
\end{eqnarray}
be a set of generators for this group and let 
$\alpha_1=(\alpha_{11},\ldots,\alpha_{1k}),\ldots,
\alpha_s=(\alpha_{s1},\ldots,\alpha_{sk})$.

\begin{theorem}\label{sept29-09-3} Let 
$Z=\{[(z_1^{\alpha_{11}}\cdots
z_k^{\alpha_{1k}},\ldots,z_1^{\alpha_{s1}}\cdots 
z_k^{\alpha_{sk}})]\vert\, z_i\in K^*\mbox{ for all
}i\}\subset\mathbb{P}^{s-1}$ be the algebraic 
toric set parameterized  by $y^{\alpha_1},\ldots,y^{\alpha_s}$. If
$\mathcal{A}=\{v_1,\ldots,v_s\}$ is homogeneous, then
\begin{enumerate}
\item[\rm (i)] $Z=V_\mathcal{A}$.
\item[\rm (ii)]
$I(V_\mathcal{A})=(y^{\alpha_1^-}t_1-y^{\alpha_1^+}z,
\ldots,y^{\alpha_s^-}t_s-y^{\alpha_s^+}z,y_1^{q-1}-1,\ldots,y_k^{q-1}-1,y_0y_1\cdots
y_k-1)\cap S$.  
\end{enumerate}
\end{theorem}

\begin{proof} (i) First we prove the inclusion ``$\subset$'': 
Take $[w]\in Z$. Let $f=t^a-t^b$ be a binomial in $I_\mathcal{A}$ with
$a=(a_i)$ and $b=(b_i)$ in $\mathbb{N}^s$. 
Notice that $f$ is homogeneous because so is $\mathcal{A}$. 
By Theorem~\ref{ipn-ufpe-cinvestav}
the ideal $I_\mathcal{A}$ is a binomial ideal. Thus we need only show
that $f(w)=0$. We can write 
$$
w=(w_i)=(z_1^{\alpha_{11}}\cdots
z_k^{\alpha_{1k}},\ldots,z_1^{\alpha_{s1}}\cdots 
z_k^{\alpha_{sk}})
$$
for some $z_1,\ldots,z_k$ in $K^*$. Let $\beta$ be a generator of the
cyclic 
group $(K^*,\,\cdot\, )$. Each $z_i$ can be written as
$z_i=\beta^{\ell_i}$ for some $0\leq\ell_i\leq q-2$. Hence 
\begin{eqnarray}
f(w)&=&(z_1^{\alpha_{11}}\cdots
z_k^{\alpha_{1k}})^{a_1}\cdots(z_1^{\alpha_{s1}}\cdots 
z_k^{\alpha_{sk}})^{a_s}-(z_1^{\alpha_{11}}\cdots
z_k^{\alpha_{1k}})^{b_1}\cdots(z_1^{\alpha_{s1}}\cdots 
z_k^{\alpha_{sk}})^{b_s}\\
&=&\beta^{p_1}-\beta^{p_2},\ \mbox{where}\nonumber\\
p_1-p_2&=&\ell_1\langle
a-b,(\alpha_{11},\ldots,\alpha_{s1})\rangle+\cdots+\ell_k\langle
a-b,(\alpha_{1k},\ldots,\alpha_{sk})\rangle.\label{sept26-09-5}
\end{eqnarray}
From 
Eqs.~(\ref{sept26-09}) and (\ref{sept26-09-1}) we have
\begin{equation}\label{sept26-09-3}
\langle(c_{j1},\ldots,c_{js}),(\alpha_{1i},\ldots,\alpha_{si})\rangle\equiv
0\mod(q-1)
\end{equation}
for all $i,j$. The difference $a-b$ is in the kernel of $A$. Thus we
can write
\begin{equation}\label{sept26-09-4}
a-b=\eta_1(c_{11},\ldots,c_{1s})+\cdots+\eta_m(c_{m1},\ldots,c_{ms})
\end{equation}
for some $\eta_i$ in $\mathbb{Z}$. If we substitute the right hand side
of Eq.~(\ref{sept26-09-4}) into Eq.~(\ref{sept26-09-5}), and then use
Eq.~(\ref{sept26-09-3}), we obtain that 
$p_1-p_2\equiv 0\mod(q-1)$. Thus $\beta^{p_1}=\beta^{p_2}$ and $f(w)=0$.

``$\supset$'': Take $[w]\in V_\mathcal{A}$. We can write 
$w=(\beta^{h_1},\ldots,\beta^{h_s})$, where $\beta$ is a generator
of the cyclic group $K^*$. Since $\mathcal{A}$ is homogeneous and
$Ac_i=0$, we get 
that $f=t^{c_i^+}-t^{c_i^-}$ is a homogeneous binomial in 
$I_\mathcal{A}$. Thus the evaluation of $f$ at $w$ is zero. This 
means that $\beta^{\langle h,c_i\rangle}=1$ for all $i$, where
$h=(h_i)$. Hence $\langle h,c_i\rangle\equiv 0\mod(q-1)$ for all $i$.
Hence using Eq.~(\ref{sept26-09}) and the choice of the $\alpha_i$'s we obtain
$$
h=\lambda_1(\alpha_{11},\ldots,\alpha_{s1})+\cdots+
\lambda_k(\alpha_{1k},\ldots,\alpha_{sk}), \ \ \ \ \lambda_i\in\mathbb{Z}.
$$
Making $z_i=\beta^{\lambda_i}$ we have 
$w=(\beta^{h_1},\ldots,\beta^{h_s})=(z_1^{\alpha_{11}}\cdots
z_k^{\alpha_{1k}},\ldots,z_1^{\alpha_{s1}}\cdots 
z_k^{\alpha_{sk}})$. Thus $[w]\in{Z}$. Part (ii) follows from (i) and
Theorem~\ref{ipn-ufpe-cinvestav-1}. 
\end{proof}

\begin{lemma}\label{sept25-09}
If $X\subset Y\subset\mathbb{T}$ and $I(X)=I(Y)$, then $X=Y$. 
\end{lemma}

\begin{proof} Let $[\alpha]=[(\alpha_i)]$ be a point in $Y$. The
ideal $\mathfrak{p}=(\{\alpha_1t_i-\alpha_it_1\}_{i=2}^s)$ is a
minimal prime of $I(Y)$, then $\mathfrak{p}$ is a minimal prime of
$I(X)$. Thus $\mathfrak{p}=(\{\gamma_1t_i-\gamma_it_1\}_{i=2}^s)$ for
some $[(\gamma_i)]\in X$. Notice that 
$\mathcal{G}_1=\{t_i-(\alpha_i/\alpha_1)t_1\}_{i=2}^s$ and 
$\mathcal{G}_2=\{t_i-(\gamma_i/\gamma_1)t_1\}_{i=2}^s$ are both
reduced Gr\"obner basis of $\mathfrak{p}$ with respect to the lex
ordering $t_s\succ\cdots\succ t_1$. Then by the uniqueness 
of such basis \cite{CLO} we obtain $\mathcal{G}_1=\mathcal{G}_2$. 
Hence $\alpha_i/\alpha_1=\gamma_i/\gamma_1$ for $i=1,\ldots,s$ and 
$(\alpha_i)=(\alpha_1/\gamma_1)(\gamma_i)$, i.e.,
$[(\alpha_i)]=[(\gamma_i)]$. This prove that $[(\alpha_i)]\in X$, as
required.
\end{proof}

\begin{proposition}\label{sept29-09-5} If $\mathcal{A}$ is homogeneous and
$\mathbb{Z}^n/\mathbb{Z}\{v_i-v_1\}_{i=2}^s$ is torsion-free, 
then $X=V_\mathcal{A}$. In particular we have equality for any $\mathcal{A}$ arising
from a connected or bipartite graph.
\end{proposition}

\begin{proof} The inclusion $X\subset V_\mathcal{A}$ is easy to see.
The ideal $I(V_\mathcal{A})$ is a graded radical ideal such that 
$t_i$ is not a zero divisor of $S/I(V_\mathcal{A})$ for all $i$. This
follows by observing the equality 
\begin{equation*}
I(V_\mathcal{A})=\cap_{[P]\in V_\mathcal{A}}I_{[P]}
\end{equation*}
where $I_{[P]}=(\alpha_1t_2-\alpha_2t_1,\alpha_1t_3-\alpha_3t_1,\ldots,
\alpha_1t_s-\alpha_st_1)$ is the prime ideal generated by the homogeneous
polynomials of $S$ that vanish on $[P]=[(\alpha_i)]$. Hence it is seen
that
$$
(I_\mathcal{A}+(t_2^{q-1}-t_1^{q-1},\ldots,t_s^{q-1}-t_1^{q-1})\colon(t_1\cdots
t_s)^{\infty})\subset I(V_\mathcal{A})\subset I(X).
$$
By Theorem~\ref{vila-dictaminadora} equality holds everywhere. Thus 
$I(V_\mathcal{A})=I(X)$. Then by Lemma~\ref{sept25-09} we get
$V_\mathcal{A}=X$. 
\end{proof}

Combining this result with Theorem~\ref{vila-dictaminadora} we
obtain:

\begin{corollary}{\rm (Finite Nullstellensatz)}\label{finite-nullstellensatz} 
If $\mathcal{A}$ is homogeneous and
$\mathbb{Z}^n/\mathbb{Z}\{v_i-v_1\}_{i=2}^s$ is torsion-free, then 
$$
(I_\mathcal{A}+(\{t_i^{q-1}-t_1^{q-1}\}_{i=2}^s)\colon(t_1\cdots
t_s)^{\infty})=
I(V(I_\mathcal{A}+(\{t_i^{q-1}-t_1^{q-1}\}_{i=2}^s))).
$$
In particular this equality holds for any $\mathcal{A}$ arising
from a connected or bipartite graph.
\end{corollary}

\section{Minimum distance in parameterized
codes}\label{minimum-distance-section}

As an application of our results, in this section we present an 
upper bound for the minimum distance of a parameterized code 
arising from   
a connected non-bipartite graph. A comparison between our bound and
the Singleton bound 
will be given. The geometric perspective of Section~\ref{vanishing-ideal-projective} plays a role
here. We will give an explicit formula for the minimum distance of
$C_X(d)$ when
$X$ is a projective torus in $\mathbb{P}^2$.

We begin with a general fact about parameterized linear codes. The
dimension of $C_X(d)$ is increasing, as a function of 
$d$, until it reaches a constant value. 
This behaviour was pointed out in \cite{duursma-renteria-tapia}
(resp. \cite{geramita-cayley-bacharach}) for finite (resp. infinite) fields. 
\begin{proposition}[\cite{duursma-renteria-tapia,geramita-cayley-bacharach}] 
Let $H_X(d)$ be the dimension of the parameterized linear code $C_X(d)$
and let $r$ be the regularity index of $S/I(X)$. Then 
$$
1=H_X(0)<H_X(1)<\cdots<H_X(r-1)<H_X(d)=|X|\ \mbox{ for }d\geq r.
$$ 
\end{proposition}

The minimum distance of $C_X(d)$ has the opposite behaviour. It is decreasing, as a
function of $d$, until
it reaches a constant value.

\begin{proposition}\label{minimum-distance-behaviour} If $\delta_d>1$
$($resp. $\delta_d=1)$, then 
$\delta_d>\delta_{d+1}$ $($resp. $\delta_{d+1}=1)$.
\end{proposition}

\begin{proof} To show the first assertion assume that $\delta_d>1$. 
For any homogeneous polynomial $F$ in $S$ we set 
$Z_X(F)=\{[P]\in X\, \vert\, F(P)=0\}$. By definition of $\delta_d$ it 
suffices to show that 
\begin{eqnarray*}
\max\{|Z_X(F)|\colon F\in
S_d;\, {\rm ev}_d(F)\neq 0\}&<&\max\{|Z_X(F)|\colon F\in
S_{d+1};\, {\rm ev}_{d+1}(F)\neq 0\},
\end{eqnarray*}
Let $F$ be a polynomial in
$S_d$ such that ${\rm ev}_d(F)\neq 0$ and with $|Z_X(F)|$ as large as
possible. As $\delta_d>1$, there
are $[P_1]\neq [P_2]$ in $X$ with $P_1=(1,a_2,\ldots,a_s)$ and 
$P_2=(1,b_2,\ldots,b_s)$ such that $F(P_i)\neq 0$ for $i=1,2$. Then
$a_k\neq b_k$ for some $k$. Let $G=F(a_kt_1-t_k)$.  Thus $G\in
S_{d+1}$, $G$ does not vanish on $X$ because $G(P_2)\neq 0$ and $G$
has more zeros than  
$F$. This proves the inequality above. The second assertion is also
easy to show. 
\end{proof}

The method of proof of the next result can also be applied to 
other families of parameterized codes, e.g., to parameterized
codes arising from Ehrhart clutters \cite{chordal} or 
from bipartite graphs.    

We come to our main application. 

\begin{theorem}\label{carlos-aron-vila} 
Let $G$ be a connected non-bipartite graph with $s$ edges, let
$V_G=\{y_1,\ldots,y_n\}$ be its vertex set, and let $X$ be the
algebraic toric set parameterized by the set of monomials $y_iy_j$ such that
$\{y_i,y_j\}$ is an edge of $G$. If $\delta_d$ is the minimum
distance of $C_X(d)$ and $d\geq 1$,  then   
$$
\delta_d\leq \left\{\begin{array}{cll}
(q-1)^{n-(k+2)}(q-1-\ell)&\mbox{if}&d\leq (q-2)(n-1)-1,\\
1&\mbox{if}&d\geq (q-2)(n-1),
\end{array}
 \right.
$$
where $k$ and $\ell$ are the
unique integers so that $k\geq 0$, $1\leq \ell\leq q-2$ and 
$d=k(q-2)+\ell$. 
\end{theorem}

\begin{proof} Let $v_1,\ldots,v_s$ be the set of all
$e_i+e_j\in\mathbb{R}^n$ such that $\{y_i,y_j\}$ is an edge of $G$. 
Thus $X$ is the algebraic toric set parameterized by
$y^{v_1},\ldots,y^{v_s}$. As $G$ is a connected non-bipartite graph, there is a
connected subgraph $H$ of $G$ with the same vertex set as $G$ and with a unique
cycle of odd length. Thus $H$ is connected non-bipartite has $n$ vertices and $n$ edges. 
We may assume that $\{v_1,\ldots,v_n\}$ is the
set of all $e_i+e_j\in\mathbb{R}^n$ such that $\{y_i,y_j\}$ is an edge of $H$. 

Consider the algebraic toric set parameterized by
$y^{v_1},\ldots,y^{v_n}$:
$$
X_1=\{[(x^{v_1},\ldots,
x^{v_n})] \vert\, x_i\in K^*\mbox{ for all }i\}\subset\mathbb{P}^{n-1}.
$$
We claim that $I(X_1)=(\{t_i^{q-1}-t_n^{q-1}\}_{i=1}^{n-1})$. Let
$B'$ be the matrix whose columns are the vectors in 
$\mathcal{B}'=\{(v_1,1),\ldots,(v_n,1)\}$. From the proof of
Corollary~\ref{sept29-09-1}, we obtain that the group
$\mathbb{Z}^{n+1}/\mathbb{Z}\mathcal{B}'$ is torsion-free, and since 
$\mathcal{B}'$ is linearly independent, using 
Corollary~\ref{sunday-morning-06-09-09}(b) we obtain
\begin{eqnarray*}
(\{t_i^{q-1}-t_n^{q-1}\}_{i=1}^{n-1})&=&(\{t_i^{q-1}-t_n^{q-1}\}_{i=1}^{n-1})\colon (t_1\cdots
t_n)^{\infty})\\
&=&(I_{\mathcal{B}'}+(\{t_i^{q-1}-t_n^{q-1}\}_{i=1}^{n-1}))\colon (t_1\cdots
t_n)^{\infty})\\ 
&=&I(X_1).
\end{eqnarray*}
This completes the proof of the claim. Let
$\mathbb{T}=\{[(x_1,\ldots,x_n)]\vert\,
x_i\in K^*\, \forall\, i\}$ be a projective torus in 
$\mathbb{P}^{n-1}$. By Corollary~\ref{sept15-09}, we have
$I(\mathbb{T})=I(X_1)$. Consequently  by Lemma~\ref{sept25-09}, we 
conclude the equality $\mathbb{T}=X_1$. 

Let $\delta_d'$ be the
minimum distance of $C_{X_1}(d)$. Next we show that
$\delta_d\leq\delta_d'$. By Corollary~\ref{sept29-09-1} one has
$|X|=|X_1|=(q-1)^{n-1}$. Therefore the projection map
$$
\overline{\theta}_1\colon X\rightarrow X_1,\ \ \ \ \ \ \ \
[(\alpha_1,\ldots,\alpha_s)]\mapsto[(\alpha_1,\ldots,\alpha_n)]  
$$
is an isomorphism of multiplicative groups. For any homogeneous
polynomial $F$, we denote its zero set by $Z_X(F)=\{[P]\in X\, \vert\, F(P)=0\}$.
Let $S'=K[t_1,\ldots,t_n]=\oplus_{d=0}^\infty S_d'$ and let $F_1\in
S_d'$ be a polynomial such that ${\rm ev}_d(F_1)\neq 0$ and with
$|Z_{X_1}(F_1)|$ as large as possible, i.e., we choose $F_1$ so that
$\delta_d'=|X_1|-|Z_{X_1}(F_1)|$. We can regard the polynomial
$F_1=F_1(t_1,\ldots,t_n)$ as an element of $S$ and denote it by $F$. 
The map $\overline{\theta}_1$
induces a bijective map
$$
\overline{\theta}_1\colon Z_X(F)\mapsto Z_{X_1}(F_1),\ \ \ \ \ \ \ \ 
[P]\mapsto \overline{\theta}_1([P]).
$$     
Therefore we have the inequality
\begin{eqnarray*}
\max\{|Z_X(F)|\colon F\in
S_d;\, {\rm ev}_d(F)\neq 0\}&\geq &\max\{|Z_{X_1}(F_1)|\colon F_1\in
S_d';\, {\rm ev}_d(F_1)\neq 0\}.
\end{eqnarray*}
Consequently  $\delta_d\leq \delta_d'$. 

Case (I): First we consider the case $1\leq d\leq (q-2)(n-1)-1$. Let 
\begin{eqnarray*}
& & M=\max\{|Z_{X_1}(F_1)|\colon F_1\in
S_d';\, {\rm ev}_d(F_1)\neq 0\},\\ 
& & \ \ \ \ \ \ M_1=(q-1)^{n-k-2}((q-1)^{k+1}-(q-1)+\ell).
\end{eqnarray*}
Next we show that $M\geq M_1$. It suffices to exhibit a homogeneous
polynomial $F_1$ in $S'$ of degree $d$ with exactly $M_1$ roots in
$X_1=\mathbb{T}$. Let $\beta$ be a generator of
the cyclic group $(K^*,\,\cdot\, )$. Consider the polynomial
$F_1=f_1f_2\ldots f_kg_\ell$, where
$f_1,\ldots,f_k,g_\ell$ are given by
\begin{eqnarray*}
f_1&=&(\beta t_1-t_2)(\beta^2 t_1-t_2)\cdots (\beta^{q-2} t_1-t_2),\\
f_2&=&(\beta t_1-t_3)(\beta^2 t_1-t_3)\cdots (\beta^{q-2} t_1-t_3),\\
\vdots&\vdots &\ \ \ \ \ \ \ \ \ \ \ \ \ \ \ \ \ \ \vdots\\
f_k&=&(\beta t_1-t_{k+1})(\beta^2 t_1-t_{k+1})\cdots (\beta^{q-2}
t_1-t_{k+1}),\\
g_\ell&=&(\beta t_1-t_{k+2})(\beta^2 t_1-t_{k+2})\cdots (\beta^{\ell}
t_1-t_{k+2}).
\end{eqnarray*}
Now, the roots of $F_1$ in $X_1$ are in one to one correspondence with the
union of the following sets: 
$$
\begin{array}{c}
\{1\}\times \{ \beta^i \}_{i=1}^{q-2} \times (K^*)^{n-2},\\
\{1\}\times \{1\} \times \{ \beta^i \}_{i=1}^{q-2} \times
(K^*)^{n-3},\\
\vdots\\
\{1\}\times \cdots \times \{1\} \times \{ \beta^i \}_{i=1}^{q-2} \times
(K^*)^{n-(k+1)},\\
\{1\}\times \cdots \times \{1\} \times \{ \beta^i \}_{i=1}^{\ell} \times
(K^*)^{n-(k+2)}.
\end{array}
$$
Therefore the number of zeros of $F_1$ in $X_1$ is given by
\begin{eqnarray*}
|Z_{X_1}(F_1)|&=&(q-2)\left[(q-1)^{n-2}+(q-1)^{n-3}+\cdots+
(q-1)^{n-(k+1)}\right]+\ell(q-1)^{n-(k+2)}\\
&=&(q-1)^{n-(k+2)}\left[(q-1)^{k+1}-(q-1)+\ell\right]=M_1,
\end{eqnarray*}
as required. Thus $M\geq M_1$. Altogether we get
\begin{eqnarray*}
\delta_d\leq \delta_d'&=&\min\{\|{\rm ev}_d(F_1)\|
\colon {\rm ev}_d(F_1)\neq 0; F_1\in S_d'\}\\
&=&|X_1|-\max\{|Z_{X_1}(F_1)|\colon F_1\in
S_d';\, {\rm ev}_d(F_1)\neq 0\}\\ 
&\leq&(q-1)^{n-1}-\left((q-1)^{n-k-2}((q-1)^{k+1}-(q-1)+\ell)\right)\\
&=&(q-1)^{n-k-2}((q-1)-\ell),
\end{eqnarray*}
where $\|{\rm ev}_d(F_1)\|$ is the number of non-zero
entries of ${\rm ev}_d(F_1)$. This completes the proof of the case 
$1\leq d\leq (q-2)(n-1)-1$. 

Case (II): Next we consider the case $d\geq (q-2)(n-1)$. Since
$I(X_1)=(\{t_i^{q-1}-t_1^{q-1}\}_{i=2}^n)$, the Hilbert series of $S'/I(X_1)$ is given by
$F_{X_1}(t)=(1-t^{q-1})^{n-1}/(1-t)^n$. Hence the regularity index of 
$S'/I(X_1)$ equals $(n-1)(q-2)$. Thus $\dim_K C_{X_1}(d)=|X_1|$
for $d\geq (n-1)(q-2)$. By the Singleton bound we get
$$
1\leq\delta_d\leq \delta_d'\leq |X_1|-\dim_K C_{X_1}(d)+1=1
$$
for $d\geq (n-1)(q-2)$. Thus $\delta_d=1$ for $d\geq
(n-1)(q-2)$.
\end{proof}

\begin{remark}\label{comparison-remark}\rm If $G$ is an odd cycle of length $n\geq
3$ and $X$ is the algebraic toric set parameterized by the edges of 
$G$, then the minimum distance of $C_X(d)$ equals
$\delta_d'$ \cite{ci-codes}. This means that for any odd cycle the bound of
Theorem~\ref{carlos-aron-vila} is sharper that the Singleton bound 
for any $d\geq 1$. For connected non-bipartite graphs which are not
cycles, our bound is sharper than the Singleton bound
within a certain range (see Example~\ref{k5}).  
\end{remark}

\begin{example}\label{k5}\rm
Let $G$ be the following complete graph on five vertices and let 
$X$ be the algebraic toric set parameterized by all $y_iy_j$ such 
that $\{y_i,y_j\}$ is an edge of $G$.
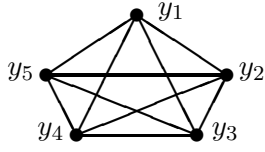
\begin{center}
$$
\setlength{\unitlength}{.040cm}
\thicklines
\begin{picture}(40,30)(0,-10)
\put(-30,0){\circle*{4.0}}
\put(30,0){\circle*{4.0}}
\put(34,0){$y_2$}
\put(0,20){\circle*{4.0}}
\put(7,20){$y_1$}
\put(0,20){\line(1,-2){20}}
\put(-30,0){\line(5,-2){50}}
\put(-43,0){$y_5$}
\put(-20,-20){\circle*{4.0}}
\put(-33,-20){$y_4$}
\put(20,-20){\circle*{4.0}}
\put(25,-20){$y_3$}
\put(-30,0){\line(3,2){30}}
\put(30,0){\line(-3,2){30}}
\put(-30,0){\line(1,-2){10}}
\put(-30,0){\line(1,0){60}}
\put(-20,-20){\line(1,0){40}}
\put(20,-20){\line(1,2){10}}
\put(-20,-20){\line(5,2){50}}
\put(-20,-20){\line(1,2){20}}
\end{picture}
$$
\end{center}

\vspace{5mm}

Let $C_X(d)$ be the parameterized code of order $d$ over the
field $K=\mathbb{F}_7$ and let 
$b_d$ (resp. $\delta_d'$) be the Singleton bound (resp. the bound of
Theorem~\ref{carlos-aron-vila}). Then the minimum distance of 
$C_X(d)$ is bounded by $\min\{b_d,\delta_d'\}$. Using {\em
Macaulay\/}$2$ \cite{mac2}, together with
Theorem~\ref{ipn-ufpe-cinvestav}, 
we obtain:   
\begin{eqnarray*}
&&\left.\begin{array}{c|c|c|c|c|c|c|c|c|c|c|c|c|c} d & 1 & 2 & 3 & 4&
5 & 6 & 7 & 8 & 9 & 10 &11&12&13 
 \\\hline
b_d & 1287 & 1252 & 1162& 977 & 646 & 316 & 127 & 36& 6 & 1 & 1
&1&1\\
\hline  
\delta_d'& 1080 & 864 & 648 & 432 & 216 & 180 & 144& 108&72 &
36&30&24&18
\end{array}\right.\\
		&&\\
		&&\left.\begin{array}{c|c|c|c|c|c|c|c} d& 14 & 15 &
		16 & 17 & 18 & 19 & 20  \\\hline \delta_d'
		&12& 6& 5 & 4 & 3 & 2 & 1
		\end{array}\right.   
\end{eqnarray*}
Thus our bound is better than the Singleton bound for $d=1,\ldots,6$.
For $d>7$ is the other way around. If $\mathbb{T}$ is a projective
torus in $\mathbb{P}^4$, it is seen that the minimum distance of
$C_\mathbb{T}(d)$ is exactly $\delta_d'$, i.e., the upper bound
$\delta_d'$ is the minimum distance of a linear code. 
\end{example}

A linear code is called
{\it maximum distance  separable\/} (MDS for short) if equality holds
in the Singleton bound. Reed-Solomon codes are MDS
\cite[p.~42]{stichtenoth}. The next result is not hard to show. 
It follows by adapting the argument of \cite[p.~42]{stichtenoth}. 

\begin{proposition} Let $\mathbb{T}=\{[(x_1,x_2)]\vert\, x_i\in 
K^*\mbox{ for }i=1,2\}$ be a projective torus in $\mathbb{P}^1$. 
Then the minimum distance $\delta_d$ of the parameterized code
$C_{\mathbb{T}}(d)$ is given by
$$
\delta_d=\left\{\hspace{-1mm}\begin{array}{cll}
q-1-d&\mbox{if}&1\leq d\leq q-3,\\
1&\mbox{if}&d\geq q-2,
\end{array}
 \right.
$$
and $C_\mathbb{T}(d)$ is an MDS code.
\end{proposition}

Finally we compute the minimum distance for the parameterized code
defined by a  
projective torus  in $\mathbb{P}^2$.

\begin{proposition}\label{minimum-distance-p2} Let
$\mathbb{T}=\{[(x_1,x_2,x_3)]\vert\, x_i\in 
K^*\mbox{ for all }i\}$ be a projective torus in $\mathbb{P}^2$. 
Then the minimum distance $\delta_d$ of the parameterized code
$C_{\mathbb{T}}(d)$ is given by

$$
\delta_d = \left\{\begin{array}{cll}
(q-1)^2- d(q-1)&\mbox{if}&1\leq d \leq q-2,\\
2q -d-3&\mbox{if}&q-1\leq d \leq 2q-5,\\
1&\mbox{if}& d\geq 2q-4.
\end{array}\right.
$$
\end{proposition}
\begin{proof} The case $1\leq d \leq q-2$ was shown in \cite[Theorem~2]{GRH}. To
show the second case assume that $q-1\leq d \leq 2q-5$. By
Corollary~\ref{sept15-09}, the vanishing ideal $I(\mathbb{T})$ is a
complete intersection generated by $t_2^{q-1}-t_1^{q-1}$ and
$t_3^{q-1}-t_1^{q-1}$. Therefore the inequality 
$\delta_{d} \geq 2q-d-3$ is a direct consequence of 
\cite[Theorem 4.4]{hansen}. Next, we write $d=(q-2)+\ell$ where
$1 \leq \ell \leq q-3$. Let $\beta$ be a generator of $(K^*,\,\cdot\,
)$. The homogeneous polynomial 
$$
F=(\beta t_{1}-t_2)\cdots (\beta^{(q-2)}
t_{1} - t_2)(\beta t_{1}-t_3)\cdots (\beta^{\ell} t_{1} - t_3)
$$
has
degree $d$ and the zero set $Z_{\mathbb{T}}(F)$ of $F$ in
$\mathbb{T}$ is the set:       
$$
\begin{array}{c}
(\{1\}\times \{ \beta^i \}_{i=1}^{q-2} \times K^*)\cup (\{1\}\times
\{1\} \times \{ \beta^i \}_{i=1}^{\ell}). 
\end{array}
$$
Therefore the number of zeros of $F$ in $\mathbb{T}$ is given by
\begin{eqnarray*}
|Z_{\mathbb{T}}(F)|&=&(q-2)(q-1)+\ell.
\end{eqnarray*}
This implies that
$$
\delta_{d}\leq(q-1)^2 - ((q-2)(q-1)+\ell)=2q-d-3.
$$
Thus $\delta_d=2q-d-3$. Finally, since the vanishing ideal of
$\mathbb{T}$ is a complete intersection, the regularity index of
$K[t_1,t_2,t_3]/I(\mathbb{T})$ is equal to $2(q-2)$. Thus by the
Singleton bound we get that $\delta_d=1$ for $d\geq 2q-4$. 
\end{proof}

The lower bound of Hansen \cite[Theorem 4.4]{hansen}---for 
the minimum distance of evaluation codes
on complete intersections---that we used in the proof above 
has been nicely generalized in \cite[Theorem~3.2]{gold-little-schenck}. 

\bibliographystyle{plain}

\end{document}